\documentclass[11pt]{amsart}

\usepackage{amssymb}
\usepackage{latexsym}
\usepackage{amsmath}
\usepackage{a4wide}
\usepackage{amsthm}
\usepackage{amsfonts}
\usepackage{color}
\usepackage{pdfsync}
\usepackage[usenames,dvipsnames]{pstricks}
\usepackage{epsfig}
\usepackage{pst-grad} 
\usepackage{pst-plot} 

\newcommand{\R}{\mathbb{R}}
\newcommand{\N}{\mathbb{N}}
\newcommand{\Z}{\mathbb{Z}}

\newcommand{\I}{\mathcal{I}}

\newcommand{\T}{\mathbb{T}}

\newcommand{\osc}{\mathrm{osc}}
\newcommand{\1}{\mathbf{1}}

\theoremstyle{plain}
\newtheorem{defi}{Definition}[section]
\newtheorem{prop}[defi]{Proposition}
\newtheorem{teo}[defi]{Theorem}

\newtheorem{lema}[defi]{Lemma}

\newtheorem{remark}[defi]{Remark}

\theoremstyle{definition}

\theoremstyle{remark}

\numberwithin{equation}{section}

\begin{document}

\title[]{Cauchy problem and periodic homogenization {for} 
 nonlocal Hamilton-Jacobi equations with coercive gradient terms}

\author[]{Martino Bardi}
\address{
Martino Bardi:
Dipartimento di Matematica Tullio Levi Civita, Universit\`a di Padova,
Via Trieste 63, Padova, Italy.
\newline {\tt bardi@math.unipd.it}
}

\author[]{Annalisa Cesaroni}
\address{
Annalisa Cesaroni:
Dipartimento di Scienze Statistiche, Universit\`a di Padova,
Via Cesare Battisti 141, Padova, Italy.
\newline {\tt annalisa.cesaroni@unipd.it}
}

\author[]{Erwin Topp}
\address{
Erwin Topp:
Departamento de Matem\'atica y C.C., Universidad de Santiago de Chile,
Casilla 307, Santiago, Chile.
\newline {\tt erwin.topp@usach.cl}
}

\date{\today}

 \begin{abstract}
 This paper deals with the periodic homogenization  of nonlocal parabolic Hamilton-Jacobi equations  with superlinear growth in the gradient terms.
We show that the problem presents  different features depending on the order of the nonlocal operator,
 giving rise to three different  cell problems  {and effective operators.  
To prove the locally uniform convergence to the unique solution of the Cauchy problem for the effective equation we need a new comparison principle among viscosity semi-solutions of integrodifferential equations that can be of independent interest.}
 \end{abstract}

\maketitle

\section{Introduction.}

This paper deals with periodic homogenization for nonlocal parabolic Hamilton-Jacobi equations of 
 the form
\begin{equation}
\label{eq}
u_t^\epsilon -a\left(x, \frac{x}{\epsilon}\right) \I(u^\epsilon, x) + H\left(x, \frac{x}{\epsilon}, Du^\epsilon\right) = 0 \quad \mbox{in} \ Q_T,
\end{equation}
where 
 $Q_T := \R^N \times (0,T]$, for $T > 0$ fixed. We complement this equation with the initial condition
\begin{equation}
\label{initial}
u^\epsilon(x,0) = u_0(x) \quad x \in \R^N,
\end{equation}
where $u_0$ is a bounded and uniformly continuous function in $\R^N$.
The elliptic part 
of the operator in \eqref{eq} is the term $a(x,y)\I(u,x)$, where $a:\R^N\times\R^N\to \R$ is a {uniformly }
continuous function and $\I(u,x)$ is a nonlocal operator defined  
as
\begin{equation}\label{operator}
\I(u,x) := \int \limits_{\R^N} [u(x + z) - u(x) - \1_B(z) \langle Du(x), z \rangle]K^\sigma(z)dz, 
\end{equation}
for  suitable functions $u: \R^N \to \R$, with $K^\sigma: \R^N \to \R$  nonnegative and measurable and $\1_B$ the indicator function of the unit ball $B$ centered at $0$. The main assumption on this nonlocal operator is the following {\it ellipticity condition}
\begin{equation}\label{aK}\tag{{\bf E}}
\begin{split}
& 0 < a_0 \leq a(x, y) \leq a_0^{-1} \quad \mbox{for all} \ x, y \in \R^N, \quad \mbox{and there exists $\sigma\in (0,2)$ such that} \\
&   \bar k(z) := K^\sigma(z)|z|^{N + \sigma} \ \mbox{is bounded in} \ \R^N, \ \mbox{continuous at the origin, and} \ \bar k(0) > 0. 
\end{split}
\end{equation}
This assumption makes $\I(u,x)$ in~\eqref{operator}  well-defined for bounded and sufficiently smooth functions $u$. The parameter $\sigma$ shall be regarded as the \textsl{order} of the operator. 

An example of particular interest is the case of the \textsl{fractional Laplacian of order $\sigma \in (0,2)$} 
%
defined as
\begin{equation}\label{fractionallap}
(-\Delta)^{\sigma/2} u(x) := -C_{N, \sigma} \ \int_{\R^N} [u(x + z)-u(x) - \mathbf{1}_B(z) \langle Du(x), z \rangle ]|z|^{-(N + \sigma)}dz,
\end{equation}
where $C_{N, \sigma} > 0$ is a suitable normalizing constant,  see~\cite{Hitch}. 

{
 We will assume $\bar k(0)=  C_{N, \sigma}$,  see assumption \eqref{ko}, so the interaction kernel  $K^\sigma$ in~\eqref{operator}    under assumption \eqref{aK} coincides with the kernel of the fractional laplacian   $(-\Delta)^{\sigma/2}$ multiplied by the function $ \frac{\bar k(z)}{\bar k(0)}$ which is bounded, continuous in $0$ and takes  value $1$ in $0$. 
So, $K^\sigma$ can be considered  a  perturbation of the  kernel of the fractional Laplacian $(-\Delta)^{\sigma/2}$, and therefore the integro-differential operator $\I$ is a perturbation of  $(-\Delta)^{\sigma/2}$.

Concerning the Hamiltonian, we concentrate here on 
 the case where 
 $H$ is superlinear 
  in the gradient variable, see assumption \eqref{H1}. 
  {   A model problem is
  \begin{equation*}
u_t^\epsilon -a\left(x, \frac{x}{\epsilon}\right) \I(u^\epsilon, x) + b\left(x, \frac{x}{\epsilon}\right) |Du^\epsilon|^m = f\left(x, \frac{x}{\epsilon}\right)
\end{equation*}
  with $m>1$ and $b(x,y)\geq b_0>0$, but we do not need any convexity of $H$ with respect to $Du$.}
  This is a suitable framework because 
  we can exploit available well-posedness and regularity results,   especially {  by Barles, Koike, Ley, and Topp \cite{BKLT},} to study the behavior of the family of viscosity solutions $\{ u^\epsilon \}_\epsilon$ to~\eqref{eq}-\eqref{initial} as $\epsilon \to 0$.    

 \smallskip

Our main purpose is to obtain homogenization results for problems of 
 the form~\eqref{eq} under  periodicity conditions on the 
  ``fast variable" $\frac{x}{\epsilon}$,  in the spirit of the celebrated paper of Lions, Papanicolaou \& Varadhan~\cite{LPV} and subsequently addressed for first and second-order degenerate elliptic and parabolic equations in~\cite{Evans1, Evans2, AB1, AB2, ABM}, among many others.
{The goal is finding an effective Hamiltonian $\bar H : \R^N\times\R^N\times \R \to \R$ such that $u^\epsilon$ converges to a solution of 
\begin{equation}
\label{eqeffective}
u_t + \bar H(x, D u, \I(u,x)) = 0 \quad \mbox{in} \ Q_T, 
 \end{equation} 
possibly the unique one satisfying the initial condition
\begin{equation}
\label{initeff}
u(x,0) = u_0(x) \quad x \in \R^N .
\end{equation}
The basic strategy to identify $\bar H$ 
begins with 
 a formal expansion in powers of $\epsilon$ of 
 the form
\begin{equation}
\label{uepsexpansionintro}
u^\epsilon(x,t) = \bar u(x,t) + \epsilon^{1 \vee \sigma} \psi(x/\epsilon),
\end{equation}
where $a \vee b = \max \{ a,b\}$ and  $\psi$ 
is called \textsl{the corrector}. Note that the exponent of $\epsilon$ is chosen depending on the order $\sigma$ of the integral operator $\I$.
Plugging the ansatz \eqref{uepsexpansionintro} 
in the equation \eqref{eq}, some nontrivial calculations in Section \ref{secexpansion} lead to a   \textsl{cell problem}, which is an an additive eigenvalue problem 
 on the torus    $\T^N$ 
  whose solution should be the corrector 
  $\psi$ and 
  the eigenvalue $\bar H=\bar H(x,p,l)$, where $x, p, l$ are parameters. 
The presence of the nonlocal term $\I$ produces three different cell problems depending on $\sigma$:
\begin{itemize}
\item for $\sigma < 1$ the cell problem is the purely first-order PDE 
\begin{equation*}
-a(x,y) l  + H(x, y, p + D\psi(y))=\bar H \qquad y\in\T^N.
\end{equation*} 

\item for $\sigma > 1$ the cell problem is the 
 linear purely nonlocal 
  equation 
\begin{equation*}
 -a(x,y) l  + a(x,y)  (-\Delta)^{\sigma/2} \psi(y) +H(x, y, p )=\bar H \qquad y\in\T^N.
\end{equation*} 

\item for $\sigma = 1$ it has both first-order and nonlocal terms, and an  extra drift term $\langle b,  D\psi(y) \rangle$
  \begin{equation}
\label{cell=1intro}
-a(x,y) l  + a(x,y) [ (-\Delta)^{1/2} \psi(y) + 
\langle b,  D\psi(y) \rangle ] + H(x, y, p+D\psi(y) )=\bar H  \quad \text{in }  \T^N,
\end{equation} 
with $b\ne 0$ if the kernel $K^1$ is not symmetric 
($b$ 
  is explicitly  defined in \eqref{J1}). 
 \end{itemize}
The solvability of these problems and sufficient regularity of $\psi$ are not difficult in the first two cases, whereas for $\sigma = 1$ they require some fine estimates that we obtain by adapting the methods of \cite{BKLT}, \cite{BLT}, and \cite{Silvestre}, and by strengthening the regularity assumption  on $H$ from the general condition \eqref{H20}  to  \eqref{A}.
We deduce from the cell problems also informations about the effective Hamiltonian $\bar H$,  especially about its modulus of continuity, since $\bar H$ is explicit only for $\sigma > 1$. 
}

Adapting in an appropriate way the 
perturbed test function method introduced  by Evans 
\cite{Evans1, Evans2}, we show 
that the weak semilimits of the family of solutions $\{ u^\epsilon \}_\epsilon$ 
are a sub- and a supersolution of} the effective equation \eqref{eqeffective} and initial condition \eqref{initeff}.
   Next we {need} 
   a comparison principle between a sub- and a supersolution of this Cauchy problem to obtain the 
locally uniform   convergence of the full sequence $\{ u^\epsilon \}_\epsilon$. 
In the nonlocal setting, {however, the known theory does not cover nonlinearities where the state variable $x$ and the integral operator $\I$ interact. Only the case $\sigma>1$, where the effective equation is
\[
u_t -  \frac{ \I(u,x)}{\int_{\T^N} 1/a(x,y)dy} + \int_{\T^N} 
{ \frac {H(x,y,Du)}{a(x,y)} }dy = 0 .
\]
can be treated by the methods of Barles and Imbert~\cite{BI}. 
For the other two cases we prove a new comparison result for \eqref{eqeffective}-\eqref{initeff} under the structure condition on the operator that for some $n>0$
\begin{equation}\label{ass-barH}
\begin{split}
& |\bar H(x_1,p_1,l_1) - \bar H(x_2, p_2, l_2)| \\
\leq & \ \omega\Big{(} |l_1 - l_2| + |x_1 - x_2|(1 +| l| + |p|^m)^n + |p_1 - p_2|(1 + |l| + |p|^m)^n \Big{)}, 
\end{split} 
\end{equation}
and for semicontinuous sub- and supersolutions attaining the initial data continuously uniformly on $\R^N$, i.e.,
\begin{equation}\label{limt0}
\lim_{t\to 0^+} \sup_{x\in\R^N} |u(x,t) - u_0(x)| = 0 ,
\end{equation}
and such that at least one of them is H\"older continuous. The proof relies on a new argument for comparison when one knows that the semisolutions are ordered in a small strip $\R^N\times [0, d_0]$ and one of them is H\"older, Proposition \ref{lematech}. Then one reduces to this case by regularising in time, and exploiting the regularity results of \cite{BKLT} and the initial condition \eqref{limt0}, see Theorem \ref{teoremac}. We believe this comparison theorem and the method of proof have independent interest and will find other applications.

Finally, we show that $\bar H$ satisfies \eqref{ass-barH} with $n=m-1$ and the weak semilimits verify the assumptions of the comparison theorem, and therefore we get the homogenization result for all $\sigma$, as well as a characterization of the limit as the unique solution of \eqref{eqeffective} with the property \eqref{limt0}.

\smallskip
There are a few other papers on the homogenization of integrodifferential equations in the framework of viscosity solutions. Arisawa \cite{A1, A2}
addressed stationary 
 equations of the form
$ u^\epsilon -a\left(\frac{x}{\epsilon}\right) \I(u^\epsilon, x) = g\left(\frac{x}{\epsilon} \right)$
  in a bounded open set $\Omega$, with $ u^\epsilon$ prescribed in $\Omega^c$. In this problem there is no interaction between $\I$ and gradient terms in $H$, and the effective equation does not depend on $x$, so it satisfies the comparison principle by standard theory. In the unpublished paper \cite{A3} she considered the same equation with the addition of a non-oscillating Hamiltonian
   $H=\max_{\alpha\in A}\langle f(x,\alpha), Du^\epsilon \rangle$,
    with $A$ compact, and mere almost periodicity of $a$ and $g$. 
In  \cite{S}   
 Schwab  also 
 considered a Dirichlet problem  and nonlocal equations without first order terms, which in his case are 
  elliptic and have the Bellman-Isaacs form 
  with oscillating kernels $K^\sigma(\frac x\epsilon,z)$. 
 In \cite{S} the effective equation 
  has nontrivial interaction between the state variable and the nonlocality, 
 but it  enjoys translation invariance properties which allow to get a comparison principle 
  by inf/sup convolutive regularizations. 
   }
Schwab also  extended some of these results to stochastic homogenization \cite{S2}.
We mention that nonlocal homogenization problems have 
 been addressed also  in other contexts, such as 
{divergence-form equations, using $\Gamma$-convergence~\cite{FBRS}, and semigroup theory \cite{PZ}.}
{Finally, we point out that a phenomenon related to the appearance of the extra term in \eqref{cell=1intro} when the kernel is not symmetric was 
observed in \cite{ChD}.}


\smallskip


The 
 paper is organized as follows. In Section \ref{sectionassumptions} we present the main assumption and preliminary results. In section~\ref{seccomparison} we provide the new comparison principle that is needed in the case $\sigma \leq 1$. In section~\ref{secexpansion} we present the different cell problems associated to the value of $\sigma \in (0,2)$.
Sections \ref{section1},  \ref{sectionminore}, and  \ref{sectionmaggiore}  deal, respectively, with the case $\sigma=1$, $\sigma<1$,  and $\sigma>1$. 
Finally, in the Appendix   we provide two 
a priori estimates for solution to coercive Hamilton-Jacobi equation with fractional Laplacian of order $1/2$.

\section{Preliminaries} \label{sectionassumptions} 

\subsection{Basic assumptions and examples.} 

First of all we assume that {$a : \R^{2N} \to \R$ 
is uniformly continuous} and $H \in C(\R^{3N})$ satisfies
\begin{equation}\label{H0}\tag{{\bf H0}}
\begin{split}
& 
| H(\cdot,\cdot,0)
|_\infty, 
| a 
|_\infty < +\infty, \\
& a(x, \cdot), \ H(x, \cdot, p) \ \mbox{are $\Z^N$-periodic, for all} \ x, p \in \R^N.
\end{split}
\end{equation}

The assumption on the nonlocal operator are given in \eqref{aK}. 
We define $\bar \omega$  to be the modulus of continuity of $\bar k$ at $0$, that is \begin{equation}\label{kernel}
\bar \omega(t) = \sup_{|z|\leq t} \{ |\bar k(z) - \bar k(0)| \}, \quad t > 0.
\end{equation} 
Moreover, in the case $\sigma=1$, we impose the following extra condition  on $K^1$, when it is not symmetric:
\begin{equation}\label{extracond1}
\int_{0}^{1} \bar \omega(r) r^{-1}dr < +\infty.
\end{equation} 

Regarding~\eqref{aK}, the second assumption is related to what we call ``the order" of the nonlocal operator, i.e., the number $\sigma\in (0,2)$. On the other hand, the first assumption is important to get the existence and uniqueness to~\eqref{eq}. 
For simplicity, we assume that 
\begin{equation}\label{ko} 
\bar k(0) = C_{N, \sigma} > 0,
\end{equation}
 where  $C_{N, \sigma} > 0$ is the well-known normalizing constant arising in the definition of fractional Laplacian  $(-\Delta)^{\sigma/2}$ (see~\cite{Hitch}). This is going to be used in subsection~\ref{secexpansion}.

\smallskip

{   
We assume  that the Hamiltonian is superlinear 
 in the gradient variable 
in the following sense: 
\begin{equation}
\label{H1}\tag{{\bf H1}}
\exists \, b_0, C_0 > 0, \;  m>1 \;:\;
\mu H(x, y, \mu^{-1}p)-H(x,y, p) \geq (1 - \mu) \Big{(} b_0 |p|^m - C_0 \Big{)}, \quad \forall \mu \in (0,1),
\end{equation}
for all $x,y,p \in \R^N$. 
Moreover, we  assume 
 there exists a modulus of continuity $\omega$ such that
\begin{equation}\label{H20}
\tag{{\bf H2}}
|H(x, y, p) -H(x', y', p')| \leq  \omega(|x - x'| + |y - y'|)(1 + R^m)  + \omega(|p - p'|)(1 + R^{m - 1}),
\end{equation}
for all $R > 0$, all $x, x', y, y'  \in \R^N$ and $p, p' \in \R^N$ with $|p|,|p'| \leq R$. 
Since it is not restrictive to assume $\omega(r)\leq C_1r$ for all $r\geq 1$, \eqref{H20} and \eqref{H0} imply the existence of $C > 0$ such that
\begin{equation}
\label{grow}
|H(x,y,p)| \leq C(1 +  |p|^m), \quad \mbox{for } \ x,y,p \in \R^{N}.
\end{equation}
We observe 
 that  assumptions {  \eqref{H0},} \eqref{H1}, and  \eqref{H20}   imply the  the  following coercivity condition: for some $C > 1$ and $K\geq 0$ 
\begin{equation}
\label{strongcond}
C^{-1}(1 +  |p|^m) -K\leq  H(x,y, p) 
, \quad \mbox{for } \ x,y,p \in \R^{N}.
\end{equation}  
A proof of this fact is detailed {at the end of the Appendix, Section \ref{appendix}}.
 A model example is 
 \begin{equation*}
H(x,y,p) = b(x,y) |p
|^m - f(x,y),
\end{equation*}
with $m > 1$ and $f, b$ bounded and  uniformly continuous, with $b \geq b_0 > 0$.
}

Finally, in the case $\sigma=1$,  we require the following extra Lipschitz condition over the data: 
recalling $m > 1$ arising in~\eqref{H1}, we assume the existence of $L > 0$ such that, for all $R > 0$, all $X = (x, y), X' = (x', y') 
 \in \R^{2N}$ and $p, p' \in \R^N$ with $|p|,|p'| \leq R$ we have 
\begin{equation}\label{A}
\left \{ \begin{array}{l}  |H(X, p) -H(X', p')| \leq L (1 + R^m) |X - X'| + L (1 + R^{m - 1}) |p - p'|,\\\\
|a(X) - a(X')| \leq L |X - X'|. \end{array} \right .
\end{equation}


   We recall briefly the definition of viscosity solutions for  nonlocal parabolic equations such as \eqref{eq}.  For more details we refer to \cite{BI}. 

\subsection{Notion of Solution.}

We describe the notion of solution for slightly more general Cauchy problems of the form
\begin{equation}\label{cau}\begin{cases} 
u_t + F(x,Du,\I(u,x)) = 0&  \mbox{in} \ Q_T\\
u(x,0)=u_0(x)& x\in \R^N.\end{cases} 
\end{equation}

{Here, $F \in C(\R^N \times \R^N \times \R)$ is \textsl{degenerate elliptic in the nonlocal variable}}, that is
\begin{equation*}
F(x,p,l_1) \leq F(x,p, l_2) \quad \mbox{for all} \quad x, p \in \R^N, \ l_1, l_2 \in \R, \ \mbox{such that} \ l_1 \geq l_2.
\end{equation*}

We introduce some notation. Let $\delta\in (0,1)$, and we denote with $B_\delta$ the ball centered at $0$ of radius $\delta$, with $B$ the ball of radius $1$,  and with $B_\delta^c$, $B^c$ the complements of such sets. Finally $B_\delta(x)$ will indicate the ball centered at  $x$ of radius $\delta$. 
 For $\phi\in C^2(\R^N\times (0, T))$ and $x\in \R^N$, $t\in(0,T)$,  we define the localized operator  
\begin{equation}\label{operatordelta} 
\I[B_\delta](\phi, x)= \int_{B_\delta} [\phi(x+z,t)-\phi(x,t)-\langle D\phi(x,t),z\rangle ]K^\sigma(z)dz.
\end{equation}  Moreover, for any $u\in L^\infty(\R^N\times(0, T))$, $p\in \R^N$ and $x\in\R^N$, $t\in(0,T)$, we define 
\begin{equation}\label{operatordelta2} 
\I[B_\delta^c](u, p,  x)= \int_{B_\delta^c} [u(x+z,t)-u(x,t)-\1_B(z)\langle p,z\rangle ]K^\sigma(z)dz.
\end{equation} 
Note that if $K^\sigma$ is symmetric, that is $K^\sigma(z)=K^\sigma(-z)$, due to its integrability properties we get that the previous operator is independent of $p\in\R^N$, 
that is 
\begin{equation}\label{operatordelta2sym} 
\I[B_\delta^c](u, p, x)=\I[B_\delta^c](u, x) =\int_{B_\delta^c} [u(x+z,t)-u(x,t) ]K^\sigma(z)dz.
\end{equation} 

\begin{defi}[Viscosity solutions]\label{defvisco} \upshape  \ \ \ \
\begin{itemize}
 \item  A bounded upper semicontinuous function  $u:\R^N\times(0, T]\to \R$ is a viscosity subsolution of \eqref{eq}  if 
 for any $(x,t)\in \R^N\times (0, T]$ and any 
test-function $\phi\in C^2(\R^N\times (0, T])$, such that $(x,t)$ is a maximum point of $u-\phi$ in   $B_\delta(x)\times (t-\delta, t+\delta)$, for a small $\delta>0$, 
there holds 
\[\phi_t(x,t)+ F \Big{(} x, D\phi(x,t), I[B_\delta](\phi, x) + I[B_\delta^c](u,D\phi(x,t), x) \Big{)} 
\leq 0. \]
 \item A bounded lower semicontinuous function  $u:\R^N\times(0, T]\to \R$ is a viscosity supersolution of \eqref{eq}  if 
 for any $(x,t)\in \R^N\times (0, T]$ and any 
test-function $\phi\in C^2(\R^N\times (0, T])$, such that $(x,t)$ is a minimum point of $u-\phi$ in   $B_\delta(x)\times (t-\delta, t+\delta)$,  for a small $\delta>0$,
there holds 
\[\phi_t(x,t)+ F \Big{(} x, D\phi(x,t), I[B_\delta](\phi, x) + I[B_\delta^c](u,D\phi(x,t), x) \Big{)} \geq 
 0. \]

\item A bounded  continuous function  $u:\R^N\times(0, T]\to \R$ is a viscosity  solution of \eqref{eq}   if it is both a subsolution and a supersolution. 
\end{itemize} 
\end{defi} 
 
\subsection{Existence and comparison principle for~\eqref{eq}-\eqref{initial}.}
In this section we present well known results about existence and uniqueness of solutions to  the Cauchy problem~\eqref{eq}-\eqref{initial}. 
We point out that we give  also a precise estimate on the behavior of the solutions to the parabolic problem  as $t\to 0$, that is estimate \eqref{initial1}, based on the 
uniform continuity assumption on the initial data,  which will be useful in comparing the weak upper and lower semilimits of $u^\epsilon$ as $\epsilon\to 0$. 
\begin{prop}
\label{existence}
Assume~\eqref{aK},~\eqref{H0},  \eqref{H1}, 
\eqref{H20}  hold and {$u_0\in BUC(\R^N)$}. Then there exists a unique bounded continuous 
viscosity solution to  the Cauchy problem~\eqref{eq}-\eqref{initial}.
Moreover, \begin{equation}\label{bound1}
|u^\epsilon|_{L^\infty(Q_T)} \leq |u_0|_\infty + |H(\cdot, \cdot, 0)|_\infty T
\end{equation}
and  
there exists a modulus of continuity $\bar \omega$ (depending on the modulus of $u_0$) such that
\begin{equation}\label{initial1}
\sup_{x \in \R^N} |u^\epsilon(x,t) - u_0(x)| \leq \bar \omega(t) \quad \mbox{for all}  \ t > 0 , \ {\epsilon >0}.
\end{equation} 
\end{prop}
\begin{proof} 
A comparison principle for bounded viscosity sub and supersolutions which are well-ordered at time $t = 0$ 
 is Proposition 3.1 in~\cite{BKLT}.  It does not apply directly to \eqref{eq} unless the coefficient $a$ multiplying the nonlocal operator $\I$ is constant.
However, in view of assumption~\eqref{aK}, equation~\eqref{eq} can be equivalently formulated as
\begin{equation*}
a^{-1}(x,x/\epsilon) u_t - \I(u,x) + a^{-1}(x, x/\epsilon)H(x, x/\epsilon, Du) = 0,
\end{equation*}
so that the nonlocal operator does not interact with the state variables {   $x, x/\epsilon$. Then,} using the continuity of $a$, we can 
{   get the comparison result by a straightforward adaption of the proof in~\cite{BKLT}. }

Concerning existence, by~\eqref{aK} and~\eqref{H0},
 if $u_0 \in C^2(\R^N)$ with $|u_0|_{C^2(\R^N)}  < \infty$, then we see that the function $U(x,t) = u_0(x) \pm C_0 t$ with $C_0$ large enough in terms of $|u_0|_{C^2(\R^N)}$ is a supersolution (resp. a subsolution) for the problem solved by $u^\epsilon$. {   More precisely, $C_0$ can be chosen of the form 
 \[
 C_0 = C_1 |D^2u_0|_\infty + C_2 |Du_0|_\infty^m ,
 \]
with $C_1, C_2$ depending only on the constants in the assumptions, 
 thanks to the linearity of $\I$ and the growth \eqref{grow} of $H$.}
So  Perron's method leads to the existence of a 
 viscosity solution to this problem. By stability arguments, it is possible to conclude the existence for initial data merely continuous by approximation. Moreover, by comparison principle the unique solution $u^\epsilon$ to problem~\eqref{eq}-\eqref{initial} is uniformly bounded in $Q_T$ for all $\epsilon > 0$, that is \eqref{bound1} holds.

We prove now \eqref{initial1}. 
If 
{   $|u_0|_{C^2(\R^N)}  < \infty$} then \eqref{initial1} holds with 
$\bar \omega(t)= C_0t$. 
In the general case, we consider a standard mollifier $\rho \in C^\infty(\R^N)$ with support in the unit ball and $\int_{B} \rho(x)dx = 1$, and its rescaled version $\rho_h(x) = h^{-N} \rho(x/h)$, $h > 0$. Then we define $u_0^h := u_0 * \rho_h$, which is a $C^\infty$ function with {   $|Du^h_0|_\infty \leq C h^{-1}$ and  $|D^2u^h_0|_\infty
 \leq C h^{-2}$}. Notice that for all $x \in \R^N$ we have
\begin{equation*}
|u_0^h(x) - u_0(x)| \leq h^{-N}\int_{B_h} |u_0(y) - u_0(x)|\rho((x - y)/h)dy \leq \omega_0(h),
\end{equation*}
where $\omega_0$ is the modulus of continuity of $u_0$. Therefore a function with the form
\begin{equation*}
U^h(x,t) = u_0^h(x)+ \omega_0(h) + C(h) t, 
\end{equation*}
is a supersolution for the problem solved by $u^\epsilon$, with a constant $C(h)$ 
   of the form
\[
C(h)=CC_1h^{-2} + CC_2h^{-m} \leq C_3 h^{-\alpha}  , \quad \alpha= 2\vee m, \quad h\leq 1
\]
where $m>1$ is the constant appearing in \eqref{H1}, \eqref{H20}. 
Since a subsolution can be constructed in the same way, we have that
\begin{equation*}
\sup_{x \in \R^N} |u^\epsilon(x,t) - u_0(x)| \leq \inf_{h > 0} \{ 2 \omega_0(h) + C(h)t \}  {   \leq 2 \omega_0(t^{\frac 1{2\alpha}}) + C_3 t^{\frac 1{2}}} =: \bar \omega(t),
\end{equation*}
which proves~\eqref{initial1} and in particular leads to~\eqref{initial}.
\end{proof} 

\section{Comparison principle and uniqueness result for a class of nonlocal Hamilton-Jacobi operators}\label{seccomparison}

In this section we provide a comparison principle among semicontinuous viscosity sub and supersolutions and a uniqueness result for problems of 
 the form~\eqref{cau}. {We need it for the effective problems addressed in Sections \ref{section1} and \ref{sectionminore} of this paper, which do not fall within the theory of \cite{BKLT}, different from 
the $\epsilon$-problem~\eqref{eq}.}

We consider the following continuity assumption: there  exists $n>0$ such that 
such that for all $x_i, p_i \in \R^N, l_i \in \R, \ i =1,2$,  
\begin{equation}\label{assF}
\begin{split}
& |F(x_1,p_1,l_1) - F(x_2, p_2, l_2)| \\
\leq & \ \omega\Big{(} |l_1 - l_2| + |x_1 - x_2|(1 +| l| + |p|^m)^n + |p_1 - p_2|(1 + |l| + |p|^m)^n \Big{)}, 
\end{split} 
\end{equation}
where $m>1$, 
 $\omega$ be a modulus of continuity, and $|p| = \max \{ |p_1|, |p_2|\},| l| = \max \{ |l_1|, |l_2|\}$.  

The initial condition $u_0\in BUC(\R^N)$  satisfies \eqref{initial1}. 

Note that  the nonlocal operator depends on the state variable: in this setting, the validity of a comparison principle among 
 semicontinuous  sub- and supersolutions  is an open problem. We provide  in Theorem \ref{teoremac}   a comparison principle  by exploiting  regularization by sup-convolutions in the time variable and the uniform continuity of  the initial datum $u_0$. We will first need a technical result for the case $\sigma=1$, which requires sufficient regularity either of the subsolution or of the supersolution, and moreover it requires to control the behavior of sub- and supersolutions in a small neighborhood of the initial time.

\begin{prop}\label{lematech}
Assume $\sigma \leq 1$. 
 Let $u,v$ bounded, $u$ u.s.c in $\bar Q_T$, $v$ l.s.c in $\bar Q_T$ be, respectively, a viscosity sub- and supersolution to the PDE in 
  \eqref{cau}, with $F$ satisfying \eqref{assF}.
Moreover we assume  \begin{equation}\label{tech}
u \leq v \quad \mbox{in} \ \R^N \times [0, d_0],
\end{equation}
for some $0 < d_0 < T$. Then there exists 
{$\alpha_0=\alpha_0(n,\sigma, m) <1$} such that, if $u$ or $v$ is in $C^\alpha(\bar Q_T)$ for some {$\alpha \in  (\alpha_0, 1)$}, 
 then $u \leq v$ in $\bar Q_T$.
\end{prop}

\noindent
{\bf \textit{Proof:}} We assume that the $C^\alpha$ property corresponds to $u$. The case in which $v$ is H\"older follows the same lines.
By contradiction, we assume that
$$
\sup_{\bar Q_T} \{ u - v\} =: M > 0.
$$

Replacing $u$ by $u - \nu t$ for some $\nu > 0$ small enough in terms of $M$ and $T$, a classical argument allows us to assume that $u$ in fact satisfies the viscosity inequality
\begin{equation*}
u_t + F(x, Du, \I(u)) \leq -\nu \quad  \mbox{in} \ Q_T.
\end{equation*}

Then,  we double variables and approximate $M$ as follows 
\begin{equation}\label{defPhi}
M_{\epsilon, \eta, \beta}=\sup_{\bar Q_T\times \bar Q_T}\Phi(x,y,s,t) := \sup_{\bar Q_T\times \bar Q_T}(u(x,s) - v(y,t) - \chi_\beta(y) - \epsilon^{-2}|x - y|^2 - \eta^{-1}(s - t)^2),
\end{equation} where the parameters  $\epsilon, \eta,\beta > 0$ are small parameters that will go to $0$, and the function $\chi_\beta$ is constructed as follows, 
arguing as  in the proof of ~\cite[Theorem 3]{BI}. We consider a function $\chi \in C^2_b(\R)$ with $\|\chi\|_{C^2} <\infty$, $\chi = 0$ in $B_1$, $\chi \geq |u|_\infty + |v|_\infty + 1$ in $B_2^c$. For $\beta > 0$ we denote $\chi_\beta(x) = \chi(\beta x)$. 

 Observe that $\chi_\beta(x)> |u|_\infty + |v|_\infty + 1$ for all $|x|\geq 2/\beta$ 
 which ensures that the supremum defining $M_{\epsilon, \eta, \beta}$ is achieved and therefore the function $\Phi$ in \eqref{defPhi}
 attains its maximum at a point $(\bar x, \bar y, \bar s, \bar t)$ for all $\beta, \epsilon, \eta > 0$ small enough.
 
Moreover, again as in ~\cite[Theorem 3]{BI} we get that \begin{equation}\label{estpsibeta}
|D\chi_\beta|_\infty, |\I(\chi_\beta, \cdot)|_\infty\to 0\qquad \text{uniformly in $\R^N$ as $\beta\to 0$}.
\end{equation}
Hence, for  $\beta> 0$ small enough in terms of $M$ we have 
\begin{equation}\label{Mbeta}
\sup_{\bar Q_T} \{ u(x,t) - v(x,t) - \chi_\beta(x) \} =: \tilde M \geq M/2,
\end{equation}
and this supremum is achieved at some point $(\hat x, \hat t) \in \bar Q_T$ with $|\hat x| \leq 2/\beta$. 
%
%
Using the inequality 
\begin{equation}\label{Phi}
\Phi(\bar x, \bar y, \bar s, \bar t) \geq \Phi(\hat x, \hat x, \hat t, \hat t) = \tilde M > 0,
\end{equation}
we see that $|\bar x - \bar y| \leq C\epsilon$ and $|\bar s - \bar t| \leq C \eta^{1/2}$. Using this and~\eqref{Phi} again together with the fact that $u$ is $C^\alpha$, we conclude that
\begin{equation*}
\tilde M \leq u(\bar y, \bar t) - v(\bar y, \bar t) + C (\epsilon^\alpha + \eta^{\alpha/2}),
\end{equation*} 
for all $\eta, \epsilon, \beta$ and a constant $C > 0$ not depending on these parameters. Then, for all $\epsilon, \eta$ small enough depending on $\tilde M$, assumption~\eqref{tech} implies that $\bar t \geq d_0$ and therefore, taking $\eta$ smaller if it is necessary, we conclude that $\bar s, \bar t \geq d_0/2$, independent of $\beta$.


Thus, we use the viscosity inequality for $u$ at $(\bar x, \bar s)$ and for $v$ at $(\bar y, \bar t)$, for each $\delta > 0$ we can write
\begin{equation}\label{testing1}
\begin{split}
2\eta^{-1}(\bar s - \bar t) + F(\bar x, \bar p, I_{\delta, 1} + I^\delta_1) & \leq -\nu \\
2\eta^{-1}(\bar s - \bar t) + F(\bar y, \bar q, I_{\delta, 2} + I^\delta_2) & \geq 0,
\end{split}
\end{equation}
where $\bar p = 2\epsilon^{-2}(\bar x - \bar y), \bar q = \bar p - D{\chi_\beta}(\bar y)$. For the nonlocal evaluations denotes as $I_{\delta, i}, I^\delta_i, \ i = 1,2$, we require some notation to split the analysis depending if $\sigma < 1$ or $\sigma = 1$. 
Denote {$\mathbb I_\sigma = 1$ if $\sigma = 1$, $\mathbb I_\sigma = 0$ if {$\sigma < 1
$}}, $\phi(x,y) := \epsilon^{-2}|x - y|^2 + {\chi_\beta}(y)$, and with this the integral terms
\begin{equation*}
\begin{split}
I_{\delta, 1} = & \int_{B_\delta} [\phi(\bar x + z, \bar y) - \phi(\bar x{, \bar y}) - \mathbb{I}_\sigma \langle \bar p, z \rangle]K^\sigma(z)dz, \\
I_{\delta, 2} = & -\int_{B_\delta} [\phi(\bar x, \bar y + z) - \phi(\bar x{, \bar y}) - \mathbb{I}_\sigma \langle \bar q, z \rangle]K^\sigma(z)dz, \\
I^\delta_1 = & \int_{B_\delta^c} [u(\bar x + z, \bar s) - u(\bar x, \bar s) - \mathbb{I}_\sigma \1_B \langle \bar p, z \rangle]K^\sigma(z)dz, \\
I^\delta_2 = & \int_{B_\delta^c} [v(\bar y + z, \bar t) - v(\bar y, \bar t) - \mathbb{I}_\sigma \1_B \langle \bar q, z \rangle]K^\sigma(z)dz,
\end{split}
\end{equation*}
where we have omitted the dependence of these quantities on the rest of the parameters for simplicity. 
Subtracting the inequalities in~\eqref{testing1}, by the continuity of $F$ and the respective semicontinuity of $u, v$ we take limit as $\eta \to 0$ to arrive at
\begin{equation}\label{testing11}
F(\bar x, \bar p, I_{\delta, 1} + I^\delta_1) -  F(\bar y, \bar q, I_{\delta, 2} + I^\delta_2) \leq -\nu, 
\end{equation}
where $\tau \in [0,T]$ is such that $\bar s, \bar t \to \tau$ as $\eta \to 0$. We keep using the notation $\bar x, \bar y$ after taking $\eta \to 0$ for simplicity.

Using that $\Phi(\bar x, \bar y, \tau, \tau) \geq \Phi(\hat x, \hat x, \hat t, \hat t)$, the definition of $\phi$ and the property of $\I$ 
 in~\eqref{estpsibeta} we arrive at
$$
I^\delta_1 \leq I^\delta_2 + o_\beta(1),
$$ 
where $o_\beta(1) \to 0$ uniformly on the rest of the parameters. 
Thus, by the elliptic monotonicity of $F$ in the nonlocal variable,~\eqref{testing11} leads us to
\begin{equation}\label{12}
F(\bar x, \bar p, I_{\delta, 1} + I^\delta_1) -  F(\bar y, \bar q, I_{\delta, 2}  + I^\delta_1 + o_\beta(1)) \leq -\nu .
\end{equation}

It is direct to check {using ~\eqref{estpsibeta} } that
\begin{equation*}
|I_{\delta,i}| \leq o_\beta(1) + C \epsilon^{-2} \left \{ \begin{array}{ll}  \delta \quad & \mbox{if} \ \sigma = 1 \\
\delta^{1 - \sigma} \quad & \mbox{if} \ \sigma < 1, \end{array} \right .
\end{equation*}
for each $i=1,2$.
On the other hand, using the $C^\alpha$ assumption for $u$ we see that  
\begin{equation*}
|I^\delta_1| \leq C\int_{B_\delta^c}|z|^{\alpha - N - \sigma}dz + C \mathbb I_\sigma |\bar p| \int_{B \setminus B_\delta} |z|^{1-N-\sigma}dz, 
\end{equation*}
from which we get
\begin{equation*}
{|I^\delta_1| \leq C \delta^{\alpha - \sigma} + C \mathbb I_\sigma |\bar p| |\log(\delta)|. }
\end{equation*}

{Next we deal first with the case $\sigma = 1$.} 
Using~\eqref{Phi} once more we see that
\begin{equation*}
u(\bar x, \tau) - u(\bar y, \tau) - \epsilon^{-2}|\bar x - \bar y|^2 \geq 0.
\end{equation*}
%
Then, applying the $C^\alpha$ continuity of $u$ we conclude that 
\begin{equation*}
C|\bar x - \bar y|^\alpha \geq \epsilon^{-2}|\bar x - \bar y|^2,
\end{equation*}
for some constant depending on $\alpha$.
%
%
From here, denoting $\theta = 2/(2 - \alpha)$ we conclude that
\begin{equation}\label{x-y1}
|\bar x - \bar y| \leq C \epsilon^{\theta}, \quad \mbox{and} \quad |\bar p| \leq C \epsilon^{\theta - 2}.
\end{equation}

Notice that $\theta \to 2$ as $\alpha \to 1^-$.




In view of the above estimates, we apply the continuity assumption on the Hamiltonian $F$ in~\eqref{12} to conclude 
\begin{equation}
\label{omega*} 
\omega \Big{(} \epsilon^{-2}\delta + o_\beta(1) + (\epsilon^\theta + o_\beta(1)) \left[1 + \delta^{\alpha - 1} + {\epsilon^{-2}\delta } +\epsilon^{\theta - 2} |\log(\delta)| + \epsilon^{m(\theta - 2)}\right]^n \Big{)} \leq -\nu 
\end{equation}

At this point we choose $\delta = \epsilon^{2+\kappa}$, $\kappa>0$ and $\beta < < \epsilon$ in order to have $o_\beta(1) = \epsilon^\theta$  to get, recalling that $m>1$ and $\theta<2$,  
\begin{equation}
\label{espeps}
\omega \Big{(} \epsilon^\kappa + \epsilon^\theta + \epsilon^{\theta+n(2+\kappa)(\alpha - 1)} + \epsilon^{\theta+nm(\theta - 2)} \Big{)} \leq -\nu 
\end{equation}
where we have replaced $\omega(\cdot)$ by $\omega(C \cdot)$ for $C > 0$ large enough. We show now that  we can choose $\kappa>0$ such that there exists a constant  $\alpha_0(n,\sigma, m  )\in (0,1)$ 
 such that for $\alpha>\alpha_0(n,\sigma, m )$ all the exponents of $\epsilon$ in \eqref{espeps}  are positive. This  will give a contradiction sending $\epsilon \to 0$ since  $\nu > 0$ is fixed. 

Indeed, choosing $\kappa=2$ and recalling that $\theta=2/(2-\alpha)$ we observe that 
\[
\theta+ 4n(\alpha-1)>0 \quad \text{ if 
 } \;\alpha\in \left(\frac{3}{2}-\frac{1}{2}\sqrt{1+\frac{2}{n}},1\right)
\]
and
\[
\theta+ mn(\theta - 2)>0\quad  \text{if and only if }\quad \alpha>1-\frac{1}{nm}.
\]
  This implies the claim {for $\sigma=1$ by} choosing 
  \[
  \alpha_0(n,1,m )=\max\left(1-\frac{1}{nm}, \frac{3}{2}-\frac{1}{2}\sqrt{1+\frac{2}{n}}\right).
\] 

{ In the case $\sigma < 1$ we argue in the same way. Now \eqref{omega*} is replaced by
\begin{equation}
\label{omega} 
\omega \Big{(} \epsilon^{-2}\delta^{1-\sigma} + o_\beta(1) + (\epsilon^\theta + o_\beta(1)) \left[ 1 + \delta^{\alpha - \sigma} + {\epsilon^{-2}\delta^{1-\sigma} } + \epsilon^{m(\theta - 2)}\right]^n \Big{)} \leq -\nu 
\end{equation}
and \eqref{espeps} is replaced by
\begin{equation*}
\omega \Big{(} \epsilon^{\kappa(1-\sigma)-2\sigma} + \epsilon^\theta + \epsilon^{\theta+n(2+\kappa)(\alpha - \sigma)} + \epsilon^{\theta+nm(\theta - 2)} \Big{)} \leq -\nu
\end{equation*}
Then we choose $\kappa> 2\sigma/(1-\sigma)$ and observe that
\[
\theta+ n(2+\kappa)(\alpha-\sigma)>0 \quad \text{ if 
  } \;\alpha\in \left(\bar \alpha
,1\right) ,
\]
where $\bar \alpha:= \left(2+\sigma-\sqrt{(2-\sigma)^2+\frac 8{n(2+\kappa)}}\right)/2$. This proves the claim for $\sigma < 1$ by choosing
  \[
  \alpha_0(n,\sigma,m )=\max\left(1-\frac{1}{nm}, \bar \alpha\right).
\] 
}
%
%
%
\qed
 
The key assumption on the regularity of the subsolution in the previous proposition can 
be obtained through the gradient dominance. We say that $F$ is {\em superlinear in the gradient} if there exist $m > 1$ and $C, c > 0$ such that
\begin{equation}
\label{superlin}
F(x,p,l) \geq c|p|^m - C(|l| + 1), \quad \mbox{for all} \ x,p \in \R^N, \ l \in \R.
\end{equation}

Now we are ready to prove a 
 comparison principle for semicontinuous solutions to problem \eqref{cau} among functions $u$ {{\em attaining uniformly continuously the initial data}, namely, satisfying
  \begin{equation}
  \label{aucid}
\sup_{x \in \R^N} | u(x,t) - u_0(x)| 
 \leq \omega_0(t), \quad t \geq 0,
\end{equation}
for some modulus $\omega_0$ (i.e., $\omega_0(t)\to 0$ as $t\to 0$.}). 
\begin{teo}
\label{teoremac}
Assume that $\sigma\leq 1$,   $F$ satisfies \eqref{assF}, it is {degenerate elliptic in the nonlocal variable and superlinear in the gradient} {\eqref{superlin}}, and 
 $u_0\in BUC(\R^N)$. 
 Let  $\underline u$ be a bounded l.s.c supersolution to \eqref{cau}  in $Q_T$ and  
 $\bar u$ be a bounded u.s.c subsolution to \eqref{cau} in $ Q_T$ {
  attaining uniformly continuously the initial data $u_0$. }
Then,  $\bar u \leq \underline u$ in $\bar Q_T$.

In particular, there exists 
{at most one} 
 viscosity solution to \eqref{cau} {among functions satisfying \eqref{aucid}.}
\end{teo}

\noindent
{\bf \textit{Proof:}} It is sufficient to prove that  $\bar u \leq \underline u$ in $\bar Q_T$, since the uniqueness of the continuous viscosity solution is a direct consequence of this. 
%
%
For $\gamma > 0$ and $(x,t) \in \bar Q_T$ we consider
\begin{equation*}
\bar u^\gamma(x,t) = \sup \limits_{s \in [0,T]} \{ \bar u(x,s) - \gamma^{-1}|s - t|^2 \}, 
\end{equation*}
and present some well-known properties for this regularization. Since $\bar u$ is u.s.c., for each $(x,t) \in \bar Q_T$, there exists $\tilde s$ depending on $x,t$ and $\gamma$ such that $\bar u^\gamma(x,t) = u(x, \tilde s) - \gamma^{-1}|t - \tilde s|^2$ and from here, noticing that $\bar u \leq \bar u^\gamma$, it is possible to conclude that $|t - \tilde s| \leq 2|u|_\infty \sqrt{\gamma}$. Using again the u.s.c. of $\bar u$, we see that $\bar u^\gamma \to \bar u$ as $\gamma \to 0$ locally uniformly in $\bar Q_T$.  

In particular, we see that for all $x$ we can write
\begin{equation*}
\bar u^\gamma(x,t) - u_0(x) \leq \bar u(x,\tilde s) - u_0(x) \leq \omega_0(t + 2|u|_\infty \sqrt{\gamma}),
\end{equation*}
where $\omega_0$ comes from \eqref{aucid}, and therefore, that for all $d > 0$ small enough, there exists $\gamma$ small in terms of $d$ such that
%
$$
\bar u^\gamma(x,t) - u_0(x) \leq \omega(d), \quad \mbox{for all} \ (x,t) \in \R^N \times [0,d],
$$ 
where $\omega$ is a modulus of continuity. 

At this point, we consider $d > 0$ fixed and define 
\begin{equation}\label{barw}
\bar w(x,t) := \bar u^\gamma(x,t) - 2(\omega(d) + \omega_0(d)).
\end{equation}
Then it is easy to see that 
$$
\bar w \leq \underline u, \quad \mbox{in} \ \R^N \times [0, d/2].
$$ 

On the other hand, standard arguments concerning sup-convolutions lead us to prove that $\bar w$ solves
\begin{equation*}
w_t + F(x, Dw, \I w) \leq 0 \quad \mbox{in} \ \R^N \times (a_\gamma, T],
\end{equation*}
where $a_\gamma > 0$ is such that $a_\gamma \to 0$ as $\gamma \to 0$. By definition the function $t \mapsto \bar u^\gamma(x,t)$ is Lipschitz continuous in $[0,T]$, uniformly in $x$ with Lipschitz constant proportional to $\gamma^{-1}$.



Due to the Lipschitz continuity of the map $t \mapsto \bar u^\gamma(x,t)$, we get that  $\bar w_t$ is bounded and {in view of the superlinear coercivity of the gradient}, we can use the H\"older estimates in~\cite[Theorem 2.1]{BKLT} to obtain $C^\alpha$ estimates for $\bar w$. In fact, if $\sigma = 1$, for each $\alpha \in (0,1)$, there exists $C$ depending on $\alpha$ and $\gamma$ such that
$$
|\bar w(x,s) - \bar w(y,t)| \leq C (|s - t| + |x - y|^\alpha), \quad \mbox{for} \ x,y \in \R^N, \ s,t \in [2a_\gamma,T].
$$ 
If $\sigma < 1$, then there exists $C > 0$ depending on $\sigma$ and $\gamma$ such that
$$
|\bar w(x,s) - \bar w(y,t)| \leq C (|s - t| + |x - y|), \quad \mbox{for} \ x,y \in \R^N, \ s,t \in [2a_\gamma,T].
$$ 
In both cases, we can fix the parameters to fulfill the requirements of Proposition~\ref{lematech}, which allows us to conclude that $\bar w \leq \underline u$ in $\R^N \times (2a_\gamma, T]$ for all $\gamma$ small enough. This implies that
\begin{equation*}
\bar u^\gamma \leq \underline u + 2(\omega(d) + \omega_0(d)) \quad \mbox{in} \  \R^N \times (2a_\gamma, T],
\end{equation*}
which implies, taking $\gamma \to 0$ that $\bar u \leq \underline u + 2(\omega(d) + \omega_0(d))$ in $\bar Q_T$. Since $d > 0$ is arbitrary, we arrive to $\bar u \leq \underline u$ in $\bar Q_T$.

 \qed

\section{The cell problems for the homogenization}\label{secexpansion} 
We consider  the  formal asymptotic expansion ~\eqref{uepsexpansionintro} and we plug it  in the equation \eqref{eq} in order to get  the effective operator, 
through the solution of the so called cell problem.

We introduce some notation.  We will denote $y=x/\epsilon$, $p=D\bar u(x,t)$, $c=-\bar u_t(x,t)$ and 
\[
l= \I(\bar u(\cdot, t),x)=\int \limits_{\R^N} [\bar u(x + z,t) - \bar u(x,t) - \1_B(z) \langle D\bar u(x,t), z \rangle]K^\sigma(z)dz. 
\]

Moreover we denote $\psi_\epsilon(x) = \psi(x/\epsilon)$ and for $e > 0$ we introduce the notation
\begin{equation}\label{defdelta}
\delta_e(v, x, z) = v(x + z) - v(x) - \1_{B_e}(z) \langle Dv(x), z \rangle,
\end{equation}
where $\1_{B_e} = \1_{B_e}^{(\sigma)}$ denotes the indicator function of $B_e$, the open ball centered at the origin with radius $e$ if $\sigma \geq 1$, and the zero function if $\sigma < 1$.

Plugging the formal asymptotic expansion ~\eqref{uepsexpansionintro} into the equation \eqref{eq}, we obtain
\begin{equation}
\label{cella1}
 -a(x,y)l-a(x,y) \epsilon^{1\vee \sigma}\I(\psi_\epsilon, x) +H(x,y, p+\epsilon^{0\vee (\sigma-1)}D\psi(y))=c.
 \end{equation}


Performing the change of variables $\xi = z/\epsilon$ we get that
\begin{align*}
\I(\psi_\epsilon, x) = 
& \epsilon^{N} \int_{\R^N} \delta_{\epsilon^{-1}} (\psi, x/\epsilon, \xi) K^\sigma(\epsilon \xi)d\xi.
\end{align*}
Using  assumption~\eqref{aK} and~\eqref{ko} we obtain 
\begin{equation}\label{split1}
\begin{split}
\I(\psi_\epsilon, x) =  \epsilon^{-\sigma}\Big{(} -(-\Delta)^{\sigma/2} \psi_\epsilon +  J (\psi_\epsilon, x)\Big{)},
\end{split}
\end{equation}
where
\begin{equation}\label{j}
J (\psi_\epsilon, x)= \int_{\R^N} \delta_{\epsilon^{-1}}(\psi, x/\epsilon, \xi) \Big{(} \bar k(\epsilon \xi) - \bar k(0) \Big{)} |\xi|^{-(N + \sigma)}d\xi.
\end{equation}

We prove now  the following claim: 
\begin{equation}\label{estJ}
\|J(\psi_\epsilon,x)\|_\infty =\begin{cases}  o_\epsilon(1) & \sigma \in (0,2), \sigma\neq 1\\ 
 \langle b, D\psi(x/\epsilon) \rangle + o_\epsilon(1), & \sigma =1, \end{cases} 
\end{equation} 
where 
\begin{equation}\label{J1}
b :=  \lim_{\rho \to 0} \int_{B \setminus B_\rho} \frac{(\bar k(z) - \bar k(0))}{|z|^{N + 1}}z dz \in \R^N,
\end{equation}
and where $o_\epsilon(1) \to 0$ as $\epsilon\to 0$ only depends on $N, \sigma$, $C^{\sigma + \alpha}$ estimates of $\psi$, $\alpha > 0$, and $\bar \omega$ in~\eqref{extracond1} when $\sigma = 1$.
Note that if {   $\bar k$ (and then $K^1$)} is symmetric, then $b=0$. This means 
that the nonlocal term develops an extra drift term when the kernel defining it is nonsymmetric
and satisfies the integrability condition \eqref{extracond1} with respect to the kernel of the square root of the Laplacian.

In order to prove the claim, we introduce some notation. For $A \subseteq \R^N$ measurable we write 
\begin{equation*}
\begin{split}
J[A] = \int_{A} \delta_{\epsilon^{-1}}(\psi, x/\epsilon, \xi) (\bar k(\epsilon \xi) - \bar k(0)) |\xi|^{-(N + \sigma)}d\xi.
\end{split}
\end{equation*}


Then, we split $J$ in~\eqref{j} as
\begin{equation*}
J = J[B] + J[B_{1/\epsilon} \setminus B] + J[B_{1/\epsilon}^c],
\end{equation*}
and we estimate each term separately. 

For $J[B]$ we perform a second-order Taylor expansion for $\psi$ in the integral term and using that $\bar k(\epsilon \xi) - \bar k(0) \leq \bar\omega(\epsilon)$ for $\xi \in B$ together with the fact that $\sigma < 2$ we arrive at
\begin{equation*}
|J[B]| \leq \frac{1}{2} \bar \omega(\epsilon) |D^2 \psi|_\infty \int_{B} |\xi|^{-N - \sigma + 2}d\xi \leq C |D^2 \psi|_\infty\bar \omega(\epsilon),
\end{equation*}
for some constant $C=C(N, \sigma) > 0$ not depending on $\epsilon$.

For $J[B_{1/\epsilon}^c]$ we notice that the compensator term  $\1_{B_{\epsilon^{-1}}}(z) \langle Du(x), z \rangle$ is no longer present in the integral and therefore we have that
\begin{equation*}
|J[B_{1/\epsilon}^c]| \leq 4|\psi|_\infty |\bar k|_\infty \int_{B_{1/\epsilon}^c} \frac{d\xi}{|\xi|^{N + \sigma}} \leq C |\psi|_\infty \epsilon^{\sigma}.
\end{equation*}


It remains to estimate $J[B_{1/\epsilon} \setminus B]$, and at this point we separate the cases $\sigma \neq 1$ and $\sigma = 1$. 

For the case $\sigma\neq 1$,  we split the remaining integral as 
$$
J[B_{1/\epsilon} \setminus B] = J[B_{1/\epsilon} \setminus B_{\theta_\epsilon}] + J[B_{\theta_\epsilon} \setminus B],
$$
where $\theta_\epsilon \to \infty$ and $\epsilon \theta_\epsilon \to 0$ as $\epsilon \to 0$. With this choice, we see that
\begin{equation*}
|J[B_{\theta_\epsilon} \setminus B]|\leq 
\begin{cases} 
\bar \omega(\epsilon \theta_\epsilon) (2|\psi|_\infty + |D\psi|_\infty) \int_{B_{\theta_{\epsilon}} \setminus B} \frac{d\xi}{|\xi|^{N + \sigma - 1}}\Big{)}\leq  C( |\psi|_\infty + |D\psi|_\infty )\bar \omega(\epsilon \theta_\epsilon), & \sigma>1\\
\bar \omega(\epsilon \theta_\epsilon)   2|\psi|_\infty \int_{B_{\theta_{\epsilon}} \setminus B} \frac{d\xi}{|\xi|^{N + \sigma}}\leq C |\psi|_\infty \bar \omega(\epsilon \theta_\epsilon)+o_\epsilon(1) &\sigma<1\end{cases} 
\end{equation*}
for some $C=C(N, \sigma)>0$ not depending on $\epsilon$. 

Similarly, for $J[B_{1/\epsilon} \setminus B_{\theta_\epsilon}]$ we have  
\[
|J[B_{1/\epsilon} \setminus B_{\theta_\epsilon}]| \leq 
\begin{cases} 2|\bar k|_\infty \left(2|\psi|_\infty      + |D\psi|_\infty \right)\int_{B_{\theta_\epsilon}^c} \frac{d\xi}{|\xi|^{N + \sigma - 1}}
\leq  C(|\psi|_\infty + |D\psi|_\infty) \theta_\epsilon^{1 - \sigma}& \sigma>1\\ 
 4|\bar k|_\infty  |\psi|_\infty \int_{B_{\theta_\epsilon}^c} \frac{d\xi}{|\xi|^{N + \sigma}} \leq C |\bar k|_\infty  |\psi|_\infty  \theta_{\epsilon}^{-\sigma}+o_\epsilon(1) & \sigma<1.\end{cases}
\]

Hence, joining the above estimates we conclude~\eqref{estJ} {   if $\sigma\ne 1$}.

\medskip
We consider now the case $\sigma=1$. First of all note that  the estimates for $J[B]$ and $J[B_{1/\epsilon}^c]$ follow the same lines above.
Moreover observe that, if $\bar k$ is symmetric, then
\[\int_{ B_{1/\epsilon} \setminus B} \langle D\psi(x/\epsilon), z \rangle \frac{(\bar k(\epsilon z)-\bar k(0))}{|z|^{N + 1}}dz=0,
 \]
therefore we can estimate $J[B_{1/\epsilon} \setminus B]$ exactly as in the case $\sigma<1$.

In the nonsymmetric case,  we consider the term $\theta_\epsilon$ present in the previous analysis for $\sigma<1$ to write
\begin{align*}
J[B_{1/\epsilon} \setminus B] = & \int_{B_{1/\epsilon} \setminus B} [\psi(x/\epsilon + \xi) - \psi(x/\epsilon)] \frac{\bar k(\epsilon \xi)-\bar k(0)}{|\xi|^{N + 1}}d\xi 
 + \int_{B_{1/\epsilon} \setminus B} \langle D\psi(x/\epsilon), \xi \rangle \frac{\bar k(\epsilon \xi)-\bar k(0)}{|\xi|^{N + 1}}d\xi \\
\leq  & C |\psi|_\infty( \bar\omega(\epsilon\theta_\epsilon)+|\bar k|_\infty \theta_\epsilon^{-1}) + \left\langle D\psi(x/\epsilon),\int_{B \setminus B_\epsilon}   \frac{\bar k(z)-\bar k(0)}{|z|^{N + 1}}zdz \right\rangle,
\end{align*}
where in the last integral we have performed the change of variables $z = \epsilon \xi$. 
We observe that, by definition \eqref{kernel} and assumption \eqref{extracond1}, 
\[ \left| \frac{\bar k(z)-\bar k(0)}{|z|^{N + 1}}z\right|\leq  \frac{\bar \omega(|z|)}{|z|^{N}} \in L^1(B).\] 
Hence,  the Dominated Convergence Theorem allows us to conclude~\eqref{estJ}. This finishes the proof of the  claim.

\medskip
Therefore, using \eqref{estJ} in \eqref{cella1}, we conclude
with different cell problems, according to the value of $\sigma$. 

\begin{description}
\item[Case $\sigma<1$] in this case \eqref{cella1} reads 
 \[-a(x,y) l +a(x,y)\epsilon^{1-\sigma} ((-\Delta)^{\sigma/2}\psi(y)+o_\epsilon(1)) + H(x, y, p + D\psi(y))=c.\]
  So the cell problem 
is the following:  for every $(x,p, l)\in \R^N\times\R^N\times \R$ there exists a unique $c=c(x,p,l)$ such that there exists a periodic viscosity solution to
\begin{equation}\label{cell<1} -a(x,y) l  + H(x, y, p + D\psi(y))=c \qquad y\in\T^N.
\end{equation} 
\item[Case $\sigma>1$] in this case   \eqref{cella1} reads 
\[-a(x,y) l +a(x,y)((-\Delta)^{\sigma/2}\psi(y)+o_\epsilon(1)) + H(x, y, p + \epsilon^{\sigma-1}D\psi(y))=c.\]   So the cell problem 
is the following: for every $(x,p, l)\in \R^N\times\R^N\times \R$ there exists a unique $c=c(x,p,l)$ such that there exists a periodic viscosity solution to
\begin{equation}\label{cell>1} -a(x,y) l  + a(x,y)  (-\Delta)^{\sigma/2} \psi(y) +H(x, y, p )=c \qquad y\in\T^N.
\end{equation} 
\item[Case $\sigma=1$] in this case   \eqref{cella1} reads 
\[-a(x,y) l +a(x,y)((-\Delta)^{\sigma/2}\psi(y)+\langle b,  D\psi(y) \rangle+o_\epsilon(1)) + H(x, y, p +  D\psi(y))=c.\]   So the cell problem 
is the following: for every $(x,p, l)\in \R^N\times\R^N\times \R$ there exists a unique $c=c(x,p,l)$ such that there exists a periodic viscosity solution to
\begin{equation}
\label{cell=1}
-a(x,y) l  + a(x,y)  {   (-\Delta)^{1/2}} \psi(y) +a(x,y) \langle b,  D\psi(y) \rangle + H(x, y, p+D\psi(y) )=c  
\end{equation} 
for $y\in\T^N$, where $b\in\R^N$ is defined in \eqref{J1} (and it is identically $0$ if $K^1$ is symmetric).
\end{description} 

\begin{remark}\upshape
Looking at the computations related to $J(\phi_\epsilon, x)$ made above in the case $\sigma = 1$, we see that if we consider nonlocal operators written in the \textsl{second order finite differences} form
\begin{equation*}
\int_{\R^N} [u(x + z) + u(x - z) - 2u(x)] K^1(z)dz
\end{equation*}
assumption~\eqref{extracond1} can be dropped. 
\end{remark}


\section{Homogenization for the case $\sigma = 1$} \label{section1} 
We start studying the cell problem introduced above. 
\begin{prop}[Cell problem]\label{propcell1}
Assume~\eqref{aK} {   with $\sigma = 1$},~\eqref{H0}, {  \eqref{H1}}, 
and~\eqref{A}. If $K^1$ is not symmetric, we additionally assume that condition~\eqref{extracond1} holds. 

Then, 
for each $x, p, l$, there exists a unique constant $c = \bar H(x,p,l)$ such that the cell problem~\eqref{cell=1}
has a classical solution $\psi \in C^{1, \alpha}$ for some $\alpha \in (0,1)$, and such solution is unique up to an additive constant.

Moreover, the following estimate holds
\begin{equation}\label{C1alpha}
|(-\Delta)^{1/2} \psi|_{L^\infty(\T^N)} \leq C (1 + |l| + |p|^m)^m,
\end{equation}
where $C > 0$ does not depend on $x, l$ nor $p$.
\end{prop}


%
%

\noindent
{\bf \textit{Proof:}} We concentrate on the case $b \neq 0$.  
Given $x, p \in \R^N$ and $l \in \R$, and for each $\delta \in (0,1)$ we consider the solution $\psi = \psi^\delta(y)$ for the approximating  problem
\begin{equation}\label{eqaprox}
\delta \psi - a(x,y) l + a(x,y) [(-\Delta)^{1/2} \psi - \langle b, D\psi \rangle ]+ H(x, y, p + D\psi) = 0, \quad y \in \T^N.
\end{equation}

The proper term $\delta \psi$ implies the existence and uniqueness of a solution $\psi^\delta$ to this problem, and the following estimate holds
\begin{equation*}
|\psi^\delta|_\infty \leq \delta^{-1} \Big{(} |a|_\infty |l|  + |H(\cdot, \cdot, p)|_\infty \Big{)},
\end{equation*}
and in view of~\eqref{A} and~\eqref{strongcond} we have the existence of a constant $C_1 > 0$ such that
\begin{equation}\label{psi1}
|\psi^\delta|_\infty \leq C_1 \delta^{-1} ( 1 +| l| + |p|^m).
\end{equation}

Then, in view of~\eqref{strongcond} and the fact that $m > 1$, it is direct to see that $\psi^\delta$ satisfies, in the viscosity sense, the inequality
\begin{equation*}
(-\Delta)^{1/2} \psi^\delta + c |D\psi^\delta|^m \leq C(1 +| l| + |p|^m) \quad \mbox{in} \ \T^N, 
\end{equation*}
from which, by applying Theorem 2.2 in~\cite{BKLT}, we get that {   $\psi^\delta$ is H\"older continuous for each exponent $\gamma \in (0,1)$.
More precisely, a careful analysis of the proof shows that there exists a constant $C_\gamma > 0$ such that
\begin{equation}
\label{cotaHolder}
|\psi^\delta(y) - \psi^\delta(y')| \leq C_\gamma(1 + \osc(\psi^\delta)^{1/m} + (|p|^m +|l|)^{1/m})|y - y'|^\gamma \quad y, y' \in \T^N.
\end{equation}
A sketch of the proof of this estimate is provided in the Appendix, Lemma~\ref{cotaHolder+}.}

From this we deduce the existence of a constant $C > 0$ such that
\begin{equation}\label{oscbound}
\osc_{\T^N} (\psi^\delta) \leq C (1 + |p|+|l|^{1/m}).
\end{equation}

%

At this point we claim that under the assumptions of the proposition together with~\eqref{psi1} and~\eqref{oscbound} we get the Lipschitz bound
\begin{equation}\label{Lipbound} 
|\psi^\delta(y) - \psi^\delta({   y'})| \leq C (1 +| l|  + |p|^m) |y - {   y'}|,
\end{equation}
for some $C > 0$ not depending on $\delta$, $x, p$ or $l$. This claim is a consequence of Theorem 3.1 in~\cite{BLT}, but we provide a proof in the appendix (Lemma~\ref{lemaLip}) for completeness.

The application of the above boundedness/regularity results in the periodic setting leads us to the solvability of the cell problem~\eqref{cell=1} by stability results of viscosity solutions by taking $\delta \to 0$. The ergodic constant is characterized as the uniform limit $\lambda = -\lim_{\delta \to 0} \delta \psi^\delta$. The uniqueness properties of the cell problem are achieved as in~\cite{Evans1} by comparison principle and strong maximum principle provided in~\cite{BKLT}. 

We devote the rest of the proof to get the $C^{1,\alpha}$ regularity. This is a consequence of a ``linearization" argument which is possible by the Lipschitz estimates given by~\eqref{Lipbound}. In fact, for a fixed $e \in \R^N$ with $|e| > 0$ we define the function 
$$
v_e(y) = (\psi(y + e) - \psi(y))/|e|.
$$

Notice that by~\eqref{Lipbound} this function $v_e$ is bounded, with 
\begin{equation}\label{ve}
|v_e|_\infty \leq C (1 + l + |p|^m).
\end{equation}

In what follows we derive an equation solved by 
 $v_e$. Using~\eqref{Lipbound} together with~\eqref{A} we get the existence of $C > 0$ such that
\begin{align*}
& |a^{-1}(x, y + e) H(x,y + e, p + D\psi(y + e)) - a^{-1}(x,y) H(x, y, p + D\psi(y))| \\ \leq & \ C(1 + l + |p|^{m})^m|e| + C(1 + l + |p|^m)^{m - 1} |D\psi_e(y)|,
\end{align*}
where {   $a^{-1}(x,y)
=1/a(x,y)$ } 
 and $\psi_e(y) = \psi(y + e) - \psi(y)$.

Using this estimate, the linearity of the fractional Laplacian, the assumptions on the data, and the uniform bounds on $v_e$, we conclude that $v_e$ satisfies, in the viscosity sense
\begin{equation*}
\begin{split}
(-\Delta)^{1/2} v_e - A(p, l) |Dv_e| & \leq C(p, l), \\
(-\Delta)^{1/2} v_e + A(p, l) |Dv_e| & \geq -C(p, l).
\end{split}
\end{equation*}
for some $A(p,l), C(p, l) > 0$ depending on the 
 parameters $p, l$ and the data, but not on $e$. 
From here, we use Theorem 6.1 in~\cite{Silvestre} (stated for parabolic problems, but easily adapted to the stationary case), or the Appendix in~\cite{DQT3}, to conclude the existence of $\alpha > 0$ (small, depending on the data and $A(p,l), C(p,l)$ but not on $e$) such that $v_e \in C^\alpha$.
This concludes the $C^{1, \alpha}$ regularity for the solution of $\psi$. 

Finally, we notice that the Lipschitz bound~\eqref{Lipbound} is inherited by 
 $\psi$ via uniform convergence. We use this into the pointwise inequality
\begin{equation*}
|(-\Delta)^{1/2}\psi(y)| \leq \lambda + C(1 + |l| + |D\psi(y)| + |H_x(y, p + D\psi)|),
\end{equation*}
which leads to~\eqref{C1alpha} using~\eqref{strongcond} and~\eqref{psi1}. This concludes the proof.
\qed

\medskip

Now we present some properties of the effective Hamiltonian. The proof is a straightforward adaptation to the corresponding effective properties given in~\cite{Evans1}.
\begin{lema}\label{lemaF1}
Let $\bar H$ be the effective Hamiltonian associated to~\eqref{cell=1}. Then
\begin{itemize}
\item[$(i)$] There exists $C > 0$ just depending on the data such that
\begin{equation*}
\begin{split}
|\bar H(x_1, p_1, l_1) - \bar H(x_2, p_2, l_2)|
\leq \ C \Big{(} &  |l_1 - l_2| + |x_1 - x_2| ( 1 + |l| + |p|^m )^m \\
& + |p_1 - p_2| (1 +| l|+ |p|^m)^{m - 1}  \Big{)},
\end{split}
\end{equation*}
where $|p| = \max \{ |p_1|, |p_2|\}$, $|l| = \max \{ |l_1|, |l_2|\}$.

\item[$(ii)$] There exists $b_0 , C > 0$ such that  for all $x, p \in \R^N, l \in \R$
$$
\bar H(x, p, l) \geq b_0 |p|^m - |a|_\infty |l| - C.
$$

\item[$(iii)$] For all $x, p \in \R^N$, the function $l \mapsto \bar H(x, p, l)$ is decreasing.
\end{itemize}
\end{lema}

\medskip
\noindent
{\bf \textit{Proof:}} 
$(i)$ Let $x_1, x_2, p_1, p_2 \in \R^N$ and $l_1, l_2 \in \R$ and for $\delta > 0$ and $i=1,2$ consider the approximating problems
\begin{equation*}
\delta \psi_i - a_i(y) l_i + a_i(y)(-\Delta)^{1/2} \psi_i + H_i(y, p_i + D\psi_i) = 0 \quad \mbox{in} \ \T^N,
\end{equation*}
where, with a slight abuse of notation we have written $a_i (y)= a(x_i,y)$ and $H_i(y, p_i + D\psi_i) = H(x_i,y, p_i + D\psi_i)$.
Then, 
we use the equation solved by $\psi_2$ and assumptions~\eqref{A} and~\eqref{H1} to write
\begin{equation*}
\begin{split}
& \delta \psi_2 - a_{1} (-\Delta)^{1/2} \psi_2 + H_{1}(y, p_1 + D\psi_2) \\
\leq & C|l_1 - l_2| + C|x_1 - x_2| \Big{(} 1 + |l| + |(-\Delta)^{1/2}\psi|_{\infty} + L_H (|p|^m + |D\psi_2|^m_{\infty}) \Big{)}\\ & + L_H|p_1 - p_2| (1 + |p|^{m - 1} + |D\psi_2|^{m - 1}_{\infty}),
\end{split}
\end{equation*}
and from this, using the Lipschitz bound~\eqref{Lipbound} and the fractional estimate~\eqref{C1alpha} we arrive at
\begin{equation}\label{ineq!}
\begin{split}
& \delta \psi_2 - a_{1} (-\Delta)^{1/2} \psi_2 + H_{1}(y, p_1 + D\psi_2) \\
\leq & C \Big{(}|l_1 - l_2| + |x_1 - x_2| ( 1 +| l| + |p|^m)^m \\
& + L_H|p_1 - p_2| ( 1 +| l| + |p|^{m})^{m - 1} \Big{)},
\end{split}
\end{equation}
for some $C > 0$ just depending on the data.


From here, by comparison it is possible to get that
\begin{equation*}
\begin{split}
\delta (\psi_2^\delta - \psi_1^\delta) \leq & C \Big{(} |l_1 - l_2| + |x_1 - x_2| ( 1 +| l| +| p|^m )^m \\ & + |p_1 - p_2| (1 + |l| + |p|^m)^{m - 1} \Big{)},
\end{split}
\end{equation*}
and a similar lower bound can be obtained. Letting $\delta \to 0^+$ and recalling the definition of $\bar H$ we conclude the result.

\medskip
\noindent
$(ii)$ We consider $\delta > 0$ and the approximating problem~\eqref{eqaprox}. Then, we consider $y_0 \in \T^N$ a maximum point to $\psi^\delta$ and using a constant function as a test function to $\psi^\delta$ at $y_0$ we can write
\begin{equation*}
\delta \psi^\delta(y_0) - a(x, y_0) l + H(x, y_0, p) \leq 0,
\end{equation*}
and using the boundedness of $a$ and coercivity of $H$ we get that
\begin{equation*}
-C - |a|_\infty |l| + b_0 |p|^m \leq -\delta \psi^\delta(y_0),
\end{equation*}
for some $C, b_0$ depending on $H$. Thus, recalling that $\delta \psi^\delta(y_0) \to -\bar H(x, p, l)$ as $\delta \to 0^+$, we conclude the result taking the limit in the right-side of the last inequality.

\medskip
\noindent
$(iii)$ We fix $x, p$, consider $l_1 < l_2$ and assume by contradiction that
\begin{equation}\label{hola1}
\bar H(x, p, l_1) < \bar H(x, p, l_2).
\end{equation}

For $i=1,2$, let $\psi_i$ solution to the cell problem
\begin{equation*}
- a(x,y) l_i + a(x,y)(-\Delta)^{1/2} \psi_i + H(x,y, p + D\psi_i) = \bar H(x, p, l_i), \quad y \in \T^N.
\end{equation*}

We can assume without loss of generality that $\psi_2 < \psi_1$. 

Next we claim that $\psi_2$ satisfies the inequality 
\begin{equation}\label{hola2}
- a(x,y) l_1 + a(x,y) (-\Delta)^{1/2}\psi_2 + H(x, y, p + D\psi_2) > \bar H(x, p, l_1)
\end{equation}
in the viscosity sense. For this, we take $y_0 \in \T^N$ and consider $\phi$ bounded and smooth such that $y_0$ is a minimum point for $\psi_2 - \phi$ in $\T^N$. Then, using the equation solved by $\psi_2$ we get
\begin{equation*}
- a(x,y_0)l_2 + a(x, y_0)(-\Delta)^{1/2} \phi(y_0) + H(x,y_0, p + D\phi(y_0)) \geq \bar H(x, p, l_2).
\end{equation*}

Then, using~\eqref{hola1}, that $l_2 > l_1$ and the nonnegativeness of $a$ we arrive at
\begin{equation*}
- a(x, y_0)l_1 + a(x, y_0)(-\Delta)^{1/2} \phi(y_0) + H(x, y_0, p + D\phi(y_0)) > \bar H(x, p, l_1),
\end{equation*}
from which the claim follows. The strict inequality in~\eqref{hola2} allows us to compare to get $\psi_2 \geq \psi_1$, which contradicts the assumed reverse inequality. This concludes the proof.
\qed

\medskip

At this point we present the main result of this section
\begin{teo}[Homogenization]
\label{teohomo1}
Under the assumptions of 
 Proposition~\ref{propcell1} {and for $u_0\in BUC(\R^N)$}, 
 the family of solutions $u^\epsilon$ of~\eqref{eq}-\eqref{initial} converges {locally} uniformly to 
{a} viscosity solution $u$ of the associated effective problem~\eqref{eqeffective} with $\bar H$ given in Proposition~\ref{propcell1}. {Moreover  $u$ is the unique 
 solution of \eqref{eqeffective} attaining uniformly continuously the initial data $u_0$.}
\end{teo}

\noindent
{\bf \textit{Proof:}} Recalling Proposition~\ref{existence}, we see that the family of functions $\{ u^\epsilon \}_\epsilon$ is uniformly bounded in $\bar Q_T$. Then, by half-relaxed limits as in~\cite{BP1} we see that the functions $\bar u = \limsup^*_\epsilon u^\epsilon$ and $\underline u = \liminf_\epsilon^* u^\epsilon$ are respective viscosity sub and supersolution to the effective problem.

To see this we argue over $\bar u$, a similar treatment can be done for $\underline u$. 
Let $(x_0, t_0) \in Q_T$ and $\phi$ be a smooth function such that $(x_0, t_0)$ is a strict global maximum point to $\bar u - \phi$. Then, for $x=x_0$, $p = D\phi(x_0)$ and $l = \I(\phi, x_0)$ let $\psi$ be a solution to~\eqref{cell=1}. In view of Proposition~\ref{propcell1} we can assume $\psi \in C^{1, \alpha}$. 

By the strict maximality of $x_0$ , the fact that $u^\epsilon \to \bar u$ locally uniformly in $\R^N$ and the boundedness of $\psi$, there exists a sequence 
$(x_\epsilon, t_\epsilon) \to (x_0, t_0)$, maximum point to $(x,t) \mapsto u^\epsilon(x,t) - (\phi(x,t) + \epsilon \psi(x/\epsilon))$ in the set $B_{R_\epsilon}(x_\epsilon) \times [0,T]$, with $R_\epsilon \to +\infty$ as $\epsilon \to 0$. 
%

Then, we can use $\phi_\epsilon(x,t)=\phi(x,t)+\epsilon \psi(x/\epsilon)$ as test function for $u^\epsilon$ at $(x_\epsilon, t_\epsilon)$ and denoting $y_\epsilon = x_\epsilon/\epsilon$ we can write
\begin{equation}\label{stgo}
\phi_t(x_\epsilon, t_\epsilon) - a(x_\epsilon, y_\epsilon)\I[B_{R_\epsilon}](\phi_\epsilon, x_\epsilon) - a(x_\epsilon, y_\epsilon)\I[B_{R_\epsilon}^c](u^\epsilon, x_\epsilon) + H(x_\epsilon, y_\epsilon, D\phi_\epsilon(x_\epsilon, t_\epsilon))\leq 0,
\end{equation}
where we have also used the notation introduced before.
By the boundedness and smoothness of $\phi$ and since $R_\epsilon \to \infty$ as $\epsilon \to 0$ we see that
\begin{equation*}
\phi_t(x_\epsilon, t_\epsilon)\to \phi_t(x_0, t_0)\qquad \I[B_{R_\epsilon}](\phi, x_\epsilon) \to \I(\phi, x_0) \quad \mbox{as} \ \epsilon \to 0,
\end{equation*}
meanwhile,  by the uniform boundedness and smoothness of $\psi$ we can use~\eqref{estJ} to conclude that
\begin{equation*}
\epsilon \I[B_{R_\epsilon}](\psi(\cdot / \epsilon), x_\epsilon) + (-\Delta)^{1/2} \psi  (y_\epsilon) - \langle b, D\psi (y_\epsilon) \rangle = o_\epsilon(1).
\end{equation*}

Plugging this into~\eqref{stgo} and using the smoothness of $\phi$ again, and the regularity assumption \eqref{A} we arrive at
\begin{multline*}
\phi_t(x_0, t_0) - a(x_0, y_\epsilon) \I(\phi, x_0) + a(x_0, y_\epsilon) (-\Delta)^{1/2} \psi(y_\epsilon)- a(x_0,y_\epsilon)  \langle b, D\psi (y_\epsilon) \rangle+
\\   H(x_0, y_\epsilon, D\phi(x_0,t_0) + D\psi (y_\epsilon)) \leq o_\epsilon(1),
\end{multline*}
and therefore
\begin{equation*}
\phi_t(x_0, t_0) + \bar H(x_0, D\phi(x_0, t_0), \I(\phi, x_0)) \leq o_\epsilon(1),
\end{equation*}
from which we conclude that $\bar u$ is a viscosity subsolution of the effective problem using the continuity of $\bar H$ and letting $\epsilon \to 0$. Observe that $\bar H$ satisfies \eqref{assF} by Lemma \ref{lemaF1}. 
 
By definition $\underline u\leq \bar u$, and moreover 
{\eqref{initial1} implies that 
$\underline u$ and $\bar u$ 
satisfy \eqref{aucid}}.  So, using Theorem \ref{teoremac}
we deduce  that $\underline u = \bar u$ in $\bar Q_T$. This concludes the proof.
\qed

\section{Homogenization in the case $\sigma < 1$} \label{sectionminore} 

We recall that when $\sigma < 1$ the compensator term {   $\1_B(z) \langle Du(x), z \rangle$} in~\eqref{operator} is not required, so    we consider in this section that
\begin{equation*}
\I(u,x) = \int_{\R^N} [u(x + z) - u(x)] K(z)dz.
\end{equation*}
Then the nonlocal operator has strictly lower order than the gradient term. In the supercritical framework given by assumption~\eqref{H1}, this leads to a dominance of the Hamiltonian term that makes the homogenization problem  
 similar to the purely first-order case  {already} 
  addressed in the literature. For this reason,  in the current section we mainly remark the new arguments involving the nonlocality. {These features would also allow to weaken some assumptions, e.g., to consider 
 kernels $K^\sigma$ {that are integrable} and with a direct dependence on $x$, but we do not pursue these generalizations here.}

\begin{prop}
[Cell problem]
\label{propcell0}
Assume~\eqref{aK} {   with $\sigma <1$}, \eqref{H0}, \eqref{H1},  \eqref{H20}. 
Then, for all $x, p \in \R^N, l \in \R$ there exists a unique constant $c = \bar H(x, p, l)$ such that problem~\eqref{cell<1} has a Lipschitz continuous viscosity solution $\psi$.
\end{prop}

As in Proposition~\ref{propcell1}, the solvability of the cell problem is obtained as the limit as $\delta \to 0$ of $\delta \psi^\delta$ with $\psi^\delta$ solving the problem
\begin{equation*}
\delta \psi^\delta(y) + H(x,y, p + D\psi^\delta(y)) - a(x,y) l = 0, \quad y \in \T^N.
\end{equation*}

The coercivity of $H$ in the gradient variable leads to the equi-Lipschitz property for the family $\psi^\delta$, see~\cite{BKLT}. 
{   Since \eqref{H0} 
gives its equiboundedness, we obtain the needed } compactness. From here, the proof follows classical lines.

%
%
%
\begin{lema}\label{propF0}
Under the assumptions of Proposition~\ref{propcell0}, 
the effective Hamiltonian $\bar H$ associated to problem~\eqref{cell<1} satisfies the 
 property
\begin{itemize}
\item[$(i')$] there exists $C > 0$ just depending on the data such that
\begin{equation*}
\begin{split}
& |\bar H(x_1, p_1, l_1) - \bar H(x_2, p_2, l_2)| \\
\leq & \ C |l_1 - l_2| + \omega(|x_1 - x_2|) (1 + | l| + {   |p|^{m}
}) + \omega(|p_1 - p_2|) (1 + {   |p|^{m -1}}), 
\end{split}
\end{equation*}
where $|l| = \max \{ |l_1|, |l_2| \}$, $|p| = \max \{ |p_1|, |p_2| \}$ and $\omega$ is a modulus of continuity related to the one in~\eqref{H20},
%
\end{itemize}
{   as well as the properties $(ii)$ and $(iii)$ of Lemma \ref{lemaF1}.}
\end{lema}

\medskip
\noindent
{\bf \textit{Proof:}} We concentrate on $(i')$ to provide explicit bounds. The proof of $(ii)$ and $(iii)$ follow as in Lemma~\ref{lemaF1}.

Let $x_1, x_2, p_1, p_2 \in \R^N$ and $l_1, l_2 \in \R$ and for $\delta > 0$ and $i=1,2$ consider the approximating problems
\begin{equation*}
\delta \psi_i - a(x_i, y) l_i + H(x_i, y, p_i + D\psi_i) = 0 \quad \mbox{in} \ \T^N.
\end{equation*}

We use the equation solved by $\psi_2$, {\eqref{H20}, 
the uniform continuity of $a$, and the known Lipschitz continuity of   $\psi_2$ ~\cite{B-CD,Barlesbook},} to write
\begin{equation*}
\begin{split}
& \delta \psi_2 - a(x_1, y) l_1 + H(x_1, y, p_1 + D\psi_2) \\
\leq & C|l_1 - l_2| + \omega(|x_1 - x_2|) (| l|  + {1} 
 + (|p| + |D\psi_2|_{\infty})^m) \\ & + \omega(|p_1 - p_2|) \Big{(}1 + (|p| + |D\psi_2|_{\infty})^{m-1}
 \Big{)},
\end{split}
\end{equation*}
where {$\omega$ is the maximum between the modulus of continuity of $a$ and the modulus appearing in  \eqref{H20}}. 
Moreover, condition~\eqref{H1} implies that $|D\psi_2|_{\infty} \leq C |p|^{1/m}$ for some $C > 0$ just depending on the data. From here, we arrive at
\begin{align*}
& \delta \psi_2 - {a(x_1, y)} 
 l_1 + H(x_1, y, p_1 + D\psi_2) \\
\leq & C|l_1 - l_2| + \omega(|x_1 - x_2|) ( 1 + |l|  +| p|^{m  \vee 1} ) + \omega(|p_1 - p_2|) \Big{(}1 + |p|^{(m - 1) \vee 0 }\Big{)},
\end{align*}
and therefore, by the comparison principle, we get the existence of $C > 0$ just depending on the data such that
\begin{equation*}
\begin{split}
\delta (\psi_2^\delta - \psi_1^\delta) \leq & C|l_1 - l_2| + \omega(|x_1 - x_2|) (1 +  |l| + |p|^{m \vee 1}) + \omega(|p_1 - p_2|) (1 +| p|^{(m - 1) \vee 0 }).
\end{split}
\end{equation*}

A similar lower bound can be obtained. Letting $\delta \to 0^+$ and considering 
 the definition of $\bar H$ we conclude the result.
\qed

\medskip

Now we are in position to prove 
the homogenization result for this case.
\begin{teo}[Homogenization]\label{teohomo0}
Under the assumptions of Proposition~\ref{propcell0} {and for $u_0\in BUC(\R^N)$}, the family of solutions $\{ u^\epsilon \}$ of~\eqref{eq}-\eqref{initial} converges {locally} uniformly to 
{a} viscosity solution $u$ of the associated effective problem~\eqref{eqeffective} with $\bar H$ given in Proposition~\ref{propcell0}. {Moreover  $u$ is the unique 
 solution of \eqref{eqeffective} attaining uniformly continuously the initial data $u_0$.}
\end{teo}

\noindent
{\bf \textit{Proof:}} As in the proof of Theorem \ref{teohomo1}, we consider the half-relaxed semilimits $\bar u, \underline u$. We are able to prove that $\bar u, \underline u$ are respective viscosity sub and supersolution to the effective problem,  the main difference being that $\phi_\epsilon$ cannot be used directly as a test function because 
$\psi$ is just Lipschitz continuous. Anyway a standard argument by contradiction based on viscosity solution theory (see \cite{AB1, 
 Evans2}) can be used to make it rigorous.

The uniqueness of the limit problem comes from Theorem \ref{teoremac}, observing that, by Lemma \ref{propF0} {and the property $\omega(r)\leq Cr$ for all $r\geq1$,} the effective operator $\bar H$ satisfies \eqref{assF} {(possibly with a different modulus $\omega$)}. 
\qed

%
%



\section{Homogenization in the case $\sigma > 1$}\label{sectionmaggiore}

In this section we deal with the case $\sigma \in (1,2)$. {Let us mention that 
the stronger ellipticity nature of this case 
would allow to weaken some assumptions, e.g., to consider non-coercive Hamiltonians $H$, but we do not pursue theses generalizations here. }

The solvability of the cell problem now reads as follows. 
 \begin{prop}[Cell problem]
 \label{propcell2}
Assume~\eqref{aK} {   with $1<\sigma<2$},~\eqref{H0}, \eqref{H1} {   and \eqref{H20}}. Then, for each $x,p,l$ there exists a constant $c = \bar H(x,p,l)$ such that the cell problem~\eqref{cell>1}
has a classical solution $C^{1, \alpha}$ with $1 + \alpha > \sigma$, and such solution is unique up to additive constants.

Moreover, we have the following characterization of the effective Hamiltonian $\bar H$:
\begin{equation}
\label{Hexplicit}
\bar H(x,p,l) = -A(x)l + \int_{\T^N} 
{   \frac {H(x,y,p)}{a(x,y)} }dy, \quad \mbox{for} \ x,p \in \R^N, \ l \in \R,
\end{equation}
where $A(x) := {   1/ \left( \int_{\T^N}
{ \frac 1{a(x,y)}}dy \right)}$. 
\end{prop} 

\noindent
{\bf \textit{Proof:}} Fixed $x,p,l$, for each $\delta > 0$ we consider the vanishing discount 
approximation of \eqref{cell>1}
\begin{equation*}
\delta \psi - a(x,y) l + a(x,y)(-\Delta)^{\sigma/2} \psi(y) + H(x, y, p) = 0, \quad y \in \T^N,
\end{equation*}
which can be uniquely solved by a function $\psi^\delta$ such that $\delta \psi^\delta$ is bounded. Then, we define the function $\tilde \psi^\delta(y) = \psi^\delta(y) - \psi^\delta(0)$ and claim that it is uniformly bounded. The argument is 
known (see for instance~\cite{B-Ch-C-I-Lip}, sublinear case), but we provide a sketch of the proof for completeness. {   By contradiction, if $\tilde \psi^\delta$ is not bounded, up to subsequences we can consider $|\tilde \psi^\delta|_\infty \to \infty$ as $\delta \to 0$ and from here we define $v^\delta = \tilde \psi^\delta/|\tilde \psi^\delta|_\infty$. By construction, $|v^\delta|_\infty = 1$ for all $\delta$ and satisfies, in the viscosity sense, a problem with the form
\begin{equation*}
-C(\delta) \leq (-\Delta)^{\sigma/2} v^\delta \leq C(\delta) \quad \mbox{in} \ \T^N,
\end{equation*}
for some constant $C(\delta) \to 0$ as $\delta \to \infty$. Then, by the interior H\"older estimates presented in~\cite{CS1} we conclude that the family $\{ v^\delta \}$ is equi-H\"older continuous. By stability results in the viscosity theory, and up to subsequences, there exists a function $\bar v$ such that $v^\delta \to \bar v$ uniformly in the torus, solving the problem $(-\Delta)^{\sigma/2} \bar v = 0$ in $\T^n$. Thus, by Strong Maximum Principle, it must be a constant. However, by construction $\bar v(0) = 0$ and $|\bar v|_\infty = 1$, a contradiction.
}

Then, using stability results over the family $\{ \tilde \psi^\delta \}$ we get the existence of a constant $c$ such that~\eqref{cell>1} has a continuous solution (which ends up to be classical by the regularity results in~\cite{Si1}). Applying a strong maximum principle in~\cite{C} we conclude this constant is unique and the solution of the problem is unique up to an additive constant.

Finally, the characterization of the effective Hamiltonian is obtained writing~\eqref{cell>1} as
\begin{equation*}
(-\Delta)^{\sigma/2} \psi = a^{-1}(x,y) (\bar H(x,p,l) - H(x,y, p)) + l =: f(x,y,p,l).
\end{equation*}

Since the fractional Laplacian is a self adjoint operator and by the strong maximum principle we have that the unique solutions to $(-\Delta)^{\sigma/2} u = 0$ in $\T^N$ are constants. By Fredholm alternative the above problem is solvable if and only if 
\begin{equation*}
\int_{\T^N} f(x,y,p,l)dy = 0,
\end{equation*}
from which the characterization of $\bar H$ follows.
\qed

The above characterization of the effective Hamiltonian allows us to conclude the homogenization result more directly.
\begin{teo}[Homogenization]
Under the assumptions of Proposition~\ref{propcell2}  {and for $u_0\in BUC(\R^N)$}, the family of solutions $u^\epsilon$ to problem~\eqref{eq}-\eqref{initial} converges {locally} uniformly to the unique viscosity solution {$u$} of the associated effective problem~\eqref{eqeffective}, with $\bar H$ given in Proposition~\ref{propcell2}, {satisfying $u(x,0)=u_0(x)$}.
\end{teo}

\noindent
{\bf \textit{Proof:}} Also in this case, the proof of the convergence of the family follows the lines provided in Theorem~\ref{teohomo1}. {Now the} uniqueness of the effective problem follows at once from the comparison principle in~\cite{BI}, noticing that the term $A$ in~\eqref{Hexplicit} is bounded and uniformly positive and therefore we can divide by it to get rid of the $x$-dependence of the nonlocality. We omit the details.

\section{Appendix}\label{appendix}

We start providing a proof for the Lipschitz bounds leading to~\eqref{Lipbound} in the proof of Proposition~\ref{propcell1}.
	\begin{lema}
	\label{lemaLip}
		Let $\delta \in (0,1)$, $\tilde a \in C(\T^N)$ strictly positive, and $\tilde H \in C(\T^N \times \R^N)$ satisfying the assumptions~\eqref{H1} (in the $x$ independent setting). 
		For $p, l$ fixed, let $\psi$ be a continuous solution to the problem
		\begin{equation*}
		\delta \tilde a \psi -l + (-\Delta)^{1/2} \psi + \tilde H(y, p + D\psi) = 0 \quad \mbox{in} \ \T^N.
		\end{equation*}
		Then there exists a constant $C > 0$ depending {   only} on the data such that
		\begin{equation*}
		|\psi(x) - \psi(y)| \leq C(1 + \osc(\psi) + |l| + |p|^m)|x - y|, \quad \mbox{for} \ x,y \in \T^N.
		\end{equation*}
	\end{lema}

\noindent
{\bf \textit{Proof:}} We follow closely the lines of Theorem 3.1 in~\cite{BLT}. We start noticing that by comparison principle, $\psi$ satisfies
$$
\delta |\psi|_\infty \leq C (1 + |l| + |\tilde H(\cdot, p)|_\infty) \leq C(1 + |l| + |p|^m),
$$
for some $C > 0$ depending on the data.

Now, replacing $\psi$ by $\psi - \inf_{\T^N} \psi + 1$ we can assume $\psi \geq 1$ at the expense of deal with the modified problem
\begin{equation*}
\delta \tilde a \psi -l + (-\Delta)^{1/2} \psi + \tilde H(y, p + D\psi) = \tilde f \quad \mbox{in} \ \T^N,
\end{equation*}
where $\tilde f(y) = \theta \tilde a(y)$ with $\theta \in \R$ satisfying $|\theta| \leq C(1 + |l| + |p|^m)$.

Then, we introduce the change of variables $\psi = e^{v}$, from which we conclude that $v$ solves the problem
\begin{equation}\label{eqJ}
\delta \tilde a -l e^{-v} - J(v,x) + e^{-v}\tilde H(y, p + e^{v}D\psi) =  \tilde f e^{-v} \quad \mbox{in} \ \T^N,
\end{equation}
where $J$ is a nonlinear nonlocal operator with the form
\begin{equation*}
J(v, x) = \int_{\R^N} [e^{v(x + z) - v(x)} - 1 - \1_B \langle Dv(x), z\rangle] |z|^{-(N + 1)}dz.
\end{equation*}

Then, for $L > 0$ we consider the function
$$
(x,y) \mapsto \Phi(x,y) = v(x) - v(y) - L|x - y|, \quad x , y \in \T^N
$$
which attains its maximum at a point $(\bar x, \bar y)$. We prove that for $L$ large enough this maximum is nonpositive from which the result follows.

By contradiction, we assume $\Phi(\bar x, \bar y) > 0$, from which $\bar x \neq \bar y$. Then we can use $\bar x$ as test point for $v$ (regarded as subsolution  to~\eqref{eqJ}) with test function $x \mapsto L|x - \bar y|$, and $\bar y$ as test point for $v$ (regarded as supersolution  to~\eqref{eqJ}) with test function $y \mapsto -L|\bar x - y|$. Substracting the viscosity inequalities and using the maximality of $(\bar x, \bar y)$ together with the definition of $J$ to control the nonlocal terms, we arrive at
\begin{equation*}
-\delta|\tilde a(\bar x) - \tilde a(\bar y)| - l (e^{-v(\bar x)} - e^{-v(\bar y)}) + \mathcal H \leq e^{-v(\bar x)}\tilde f(\bar x) - e^{-v(\bar y)} \tilde f(\bar y),
\end{equation*}
where 
$$
\mathcal H = e^{-v(\bar x)}\tilde H(\bar x, p + L e^{v(\bar x)}\hat a) - e^{-v(\bar x)}\tilde H(\bar x, p + L e^{v(\bar x)}\hat a),
$$
and $\hat a = (x - y)/|x - y|$.

From now on we denote $\mu = e^{v(y) - v(x)} \in (0,1)$. By the assumptions, the last inequality lead us to
\begin{equation}\label{ciao}
-L_{\tilde a} |\bar x - \bar y| - l^+e^{-v(y)} (1 - \mu) + \mathcal H \leq e^{-v(y)} (|\tilde f| (1 - \mu) + L_{\tilde f}|x - y|),
\end{equation}

From here we focus on $\mathcal H$. Notice that
\begin{equation*}
\mathcal H = e^{-v(y)} \Big{(} \mu \tilde H(x, p + \mu^{-1} L \tilde p) - \tilde H(y, p + L \tilde p)\Big{)}, \quad \tilde p = e^{v(\bar y)}\hat a.
\end{equation*}

In view of the assumption on $\tilde H$ we see that
\begin{align*}
\mathcal H \geq -C e^{-v(y)} \Big{\{}&  -L_H (1 + |p + L \mu^{-1} \tilde p|^m)|x - y| \\
& - L_H (1 + |p + L\tilde p|^{m - 1})|(1 - \mu)|p| \\
& + (1 - \mu)(c|\mu p + L \tilde p|^m - C) \Big{\}}
\end{align*}

If we assume that $L \geq \max \{ 1, 4|p|, l^+, |\tilde f|_\infty, Cc^{-1} \}$ we can write
\begin{align*}
\mathcal H \geq (1 - \mu)cL^m |\tilde p|^m,
\end{align*} 
for some constants $C,c > 0$. Hence,~\eqref{ciao} reduces to
\begin{align*}
-L_{\tilde a} |\bar x - \bar y| + c(1 - \mu)e^{-v(y)}L^m |\tilde p|^m \leq e^{-v(y)} L_{\tilde f}|x - y|.
\end{align*}

At this point, we notice that the maximality of $(x,y)$ we see that $L|x - y| \leq v(x) - v(y)$, which in turn implies that $|x - y| \leq L^{-1}\mathrm{osc}(v)$. Then, considering additionaly $L$ large enough in terms of $\mathrm{osc}(v)$  ($L \geq 2\mathrm{osc}(v)$), by definition of $\mu$ we can conclude that
\begin{align*}
1 - \mu \geq 1 - e^{-L|x - y|} \geq e^{-\osc(v)}L|x - y|.
\end{align*}

Using this and cancelling the common factor $|x - y| > 0$ in the last inequality, and using the definition of $\tilde p$ we arrive at
\begin{align*}
-L_{\tilde a}  + ce^{-osc(v)}e^{(m-1)v(y)}L^{m + 1} \leq e^{-v(y)} L_{\tilde f}.
\end{align*}

Then, since $m > 1$ we get that additionally assuming that 
$$
L \geq C\max \{ c^{-1}e^{osc(v)}, L_a, L_f \}
$$ 
for a large universal constant $C > 1$ we arrive at a contradiction. Finally, recalling the relation $\psi = e^v$ we notice that 
\begin{equation*}
e^{osc (v)} = \frac{e^{\sup v}}{e^{\inf v}} = \frac{\sup \psi - \inf \psi}{e^{\inf v}} + 1 \leq osc(\psi) + 1,
\end{equation*}
from which the dependence on the oscillation is obtained.
\qed

\medskip

Next, we provide a sketch of the proof of~\eqref{cotaHolder}, presented as the following
\begin{lema}\label{cotaHolder+}
Let $c_0, C_0 > 0$, $m > 1$ and $u$ be a bounded, upper semicontinuous viscosity solution to the problem
\begin{equation*}
(-\Delta)^{1/2} u + c_0 |Du|^m \leq C_0 \quad \mbox{in} \ \T^N.
\end{equation*}

Then, for each $\gamma \in (0,1)$, there exists a constant $C_\gamma > 0$ just depending on $\gamma, m, N$ and $c_0$ such that
\begin{equation*}
|u(x) - u(y)| \leq C_\gamma \Big{(} 1 + (\osc(u) + C_0)^{1/m} \Big{)}  |x-y|^\gamma  \quad\forall\, x, y \in \T^N.
\end{equation*}
\end{lema}

\noindent
{\bf \textit{Proof:}} Fix $\gamma \in (0,1)$ and let $x_0 \in \T^N$. 
We look for a constant $L > 0$ large enough, not depending on $x_0$, such that
\begin{equation*}
u(x) - u(x_0) \leq L |x - x_0|^{ \gamma} \quad \mbox{for all} \ x \in \T^N.
\end{equation*}
We proceed by contradiction. Then, for every $L>0$ there exists $\theta_L>0$ and  $\bar x \in \T^N, \bar x \neq x_0$ such that
\begin{equation*}
u(\bar x) - u(x_0) - L|\bar x - x_0|^\gamma = \max_{x \in \T^N} \{ u(x) - u(x_0) - L|x - x_0|^\gamma \} \geq \theta_L.
\end{equation*}
Now, we observe that we can use the function $\phi(x) =  L|x - x_0|^\gamma$ as test function for $u$ at $\bar x$. 
Actually we fix $\delta_0<|\bar x-x_0|$ 
and we consider a smooth function  $\phi_0$ 
which coincides with $\phi$ in
$B(\bar x,\delta_0)$. So $u-\phi_0$ has a maximum at $\bar x$  in $B_{\delta_0}(\bar x)$ 
and  recalling Definition \ref{defvisco}, we get that 
for any $0<\delta<\min (1, \delta_0)$, there holds 
\begin{equation}\label{last}
-\I[B_\delta](\phi, \bar x) - \I[B_\delta^c](u,  \bar x) + c_0 \gamma^m L^m |\bar x - x_0|^{m(\gamma - 1)} \leq C_0,
\end{equation}
where 
$-\I = (-\Delta)^{1/2}$, and $\I[B_\delta](\phi, \bar x)$ and $\I[B_\delta^c](u,\bar x)$ have been defined in \eqref{operatordelta} and \eqref{operatordelta2sym}.

Using the fact that $\bar x$ is a maximum point to $u - \phi$ we can write
\begin{equation*}
\I[B_\delta^c](u, \bar x) \leq  \I[B \setminus B_\delta](\phi, \bar x) + \I[B^c](u, \bar x).\end{equation*}
Now it is easy to check that $ \I[B^c](u, \bar x)\leq  C \ \osc(u)$ for some universal constant $C>0$. 
Moreover, recalling the definition of $(-\Delta)^{1/2}$ in \eqref{fractionallap}, of $\phi$ and $\delta$, we get that for any $\delta\in (0, \min(1, \delta_0))$, 
\begin{multline*}
\I[B \setminus B_\delta](\phi, \bar x)=L C_{N, 1}  \int_{\delta<|z|<1} [|\bar x +z- x_0|^\gamma-|\bar x - x_0|^\gamma] |z|^{-N-1}dz\\ \leq 
L C_{N, 1} \int_{\delta<|z|<1} |z|^\gamma |z|^{-N-1}dz\leq CL(\delta^{\gamma-1}-1)\leq CL(|\bar x-x_0|^{\gamma-1}-1)\
\end{multline*} 
for some  constant $C>0$, depending on $N$ and $\gamma$. 
On the other hand, if we fix  $\delta= |\bar x - x_0|/2$,  we observe that there exists a constant $C_0>0$ depending only  on $\gamma$ and $N$ such that 
for every $z\in B_\delta$ we get 
\begin{equation*}
\phi(\bar x + z) - \phi(\bar x) - \langle D\phi(\bar x), z \rangle = \frac{1}{2} \int_{0}^{1} (1 - t)\langle D^2\phi(\bar x + tz) z, z \rangle dt\leq C_0L|\bar x-x_0|^{\gamma-2}|z|^2. 
\end{equation*} 
Therefore, we conclude  that 
\begin{equation*}
\I[B_\delta](\phi, \bar x) \leq C_{N, 1} C_0 L  |\bar x - x_0|^{\gamma - 2} \int_{B_\delta} |z|^2 |z|^{-(N + 1)}dz \leq C L |\bar x - x_0|^{\gamma - 1}.
\end{equation*} for some constant $C>0$ depending only on $N, \gamma$. 

Joining the above estimates into~\eqref{last} we can write
\begin{equation*}
c_0\gamma^m L^m |\bar x - x_0|^{m(\gamma - 1)} \Big{(} 1- \frac{C}{c_0} \gamma^{-m}(L |\bar x - x_0|^{\gamma - 1})^{1 - m}\Big{)} \leq  C \ \osc(u) + C_0,
\end{equation*}
and since we can assume $|\bar x - x_0| \leq \sqrt{N}$ together with the fact that $m > 1$, we arrive at
\begin{equation*}
c_0\gamma^m L^m |\bar x - x_0|^{m(\gamma - 1)} \Big{(} 1 - \frac{C}{c_0}\gamma^{-m} (L N^{(\gamma - 1)/2})^{1 - m}\Big{)} \leq  C \ \osc(u) + C_0.
\end{equation*}

Thus, taking $L$ large enough in terms of $c_0, N, \gamma$ and $m$ we arrive at 
$$
c_0 \gamma^{m} L^m |\bar x - x_0|^{m(\gamma - 1)} \leq 2(C\ \osc(u) + C_0),
$$ 
from which we arrive at a contradiction by taking $L$ sufficiently large in terms of $C_0/c_0$.
\qed

\medskip

We finish with the proof of the following

\medskip

{  
\noindent
{\bf \textit{Claim:}} \textsl{Conditions {  \eqref{H0},} \eqref{H1} and \eqref{H20} imply~\eqref{strongcond}.}

\medskip

By uniform continuity of $H$, see assumption \eqref{H20}, 
{  from   \eqref{H0} we can get }
~\eqref{strongcond} in the case of $p$ bounded,  by taking $K > 0$ large enough. So, we take $K>0$ large enough such that \eqref{strongcond} holds for all $|p|\leq 2$.  Thus, from here we concentrate on the case of  $|p|>2$.

Consider $q \in \R^N, \ q \neq 0$. Now, for $\mu \in (0,1)$ and  $k\in \N$,  applying \eqref{H1} with $p=\mu^{-(k-1)}q$, we have
\begin{equation*}
\mu H(x,y,\mu^{-k} q) - H(x,y, \mu^{-(k - 1)} q) \geq (1 - \mu) \Big{(}  b_0 \mu^{-m(k - 1)}|q|^m - C_0 \Big{)}.
\end{equation*}
We multiply the above inequality by $\mu^{k-1}$ and sum it up from $k=1$ to $n$ for some $n \in \N$,  and we conclude that 
\begin{multline}
\label{uno}
\mu^n H(x,y,\mu^{-n}q) - H(x,y, q)= \sum_{k=1}^n \left[\mu^k H(x,y,\mu^{-k} q) - \mu^{k-1} H(x,y, \mu^{-(k - 1)} q)\right]\\
 \geq \sum_{k=1}^n\left[(1 - \mu) \left( b_0 \mu^{(1-m)(k - 1)}|q|^m - C_0\mu^{k-1}\right)\right]
 =  (1 - \mu) b_0 |q|^m \frac{ \mu^{n(1 - m)}-1}{\mu^{1 - m}-1} - C_0 (1 - \mu^{n}).
  \end{multline}
Let fix  $|p| > 2$ and let $n\in \N$ such that $2^{n}\leq |p|\leq 2^{n+1}$.  Let $\mu=|p|^{-\frac{1}{n}}<1$. 
 Note that  by our choice of $n$, $\mu\in \left[\frac{1}{4}, \frac{1}{2}\right]$. 
Then, from \eqref{uno},  applied to $q=\frac{p}{|p|}$ and to $\mu$ and $n$ as above, so that $p=\mu^{-n}q$, we get that
\begin{equation}\label{due} 
|p|^{-1} H(x,y, p) -  H(x,y, q) \geq (1 - \mu) b_0 \frac{|p|^{m-1}-1}{\mu^{1-m}-1} 
- C_0 \left(1-|p|^{-1}\right).
\end{equation}
Observe that $\frac{1}{2}\leq 1-\mu\leq \frac{3}{4}$ and 
$$
\frac{1}{\frac{1}{4}^{1-m}-1}\leq \frac{1}{\mu^{1-m}-1}\leq \frac{1}{\frac{1}{2}^{1-m}-1}.
$$ 
So there exist  constants $c_m, C_m>0$ depending only on $m>1$ such that 
 \[c_m\leq (1 - \mu) \frac{1}{\mu^{1-m}-1}\leq C_m.\]  
Therefore from \eqref{due} we get 
\[
 H(x,y,p)\geq b_0c_m |p|^m- b_0C_m|p| +|p|H(x,y,q)-C_0(|p|-1).
\]
By \eqref{grow} we have that $H(x,y,q)\geq -C$ for all $q\in\R^N$ with $|q|=1$, so,
we conclude that 
\[ H(x,y,p)\geq b_0c_m |p|^m-|p| (b_0C_m+C+C_0)+C_0\qquad \forall |p|>2.\] Therefore, recalling that $m>1$, we conclude that there exist $\tilde C>0$, and $K>0$, depending on $m, b_0, C_m, c_m, C, C_0$ 
such that \[ H(x,y,p)\geq \tilde C (|p|^m+1)-K\qquad \forall |p|>2.\] 
\qed
}

\bigskip

\noindent
{\textbf{Aknowledgements:}  E.T. was partially supported by Conicyt PIA Grant No. 79150056, Foncecyt Iniciaci\'on No 11160817 and by the Visiting Professor grant of the Department of Mathematics ``Tullio Levi-Civita'' of the University of Padova. 
M. B. and A. C.  were partially supported by the research projects  ``Mean-Field Games and Nonlinear PDEs" of the University of Padova {  and ``Nonlinear Partial Differential Equations: Asymptotic Problems and Mean-Field Games" of the Fondazione CaRiPaRo}. M. B. and A. C. are members of the Gruppo Nazionale per l'Analisi Matematica, la Probabilit\`a e le loro Applicazioni (GNAMPA) of the Istituto Nazionale di Alta Matematica (INdAM).



\begin{thebibliography}{00}
\bibitem{AB1}
O. Alvarez, M. Bardi, {\em Viscosity Solutions Method for Singular Perturbations in Deterministic and Stochastic Control.} SIAM J. Control Optim.  40, 
1159-1188 (2001)

\bibitem{AB2}
O. Alvarez, M. Bardi, {\em Ergodicity, Stabilization, and Singular Perturbations for Bellman-Isaacs Equations.} Mem. Am. Math. Soc. 204 (960) (2010).

\bibitem{ABM}
O. Alvarez, M.  Bardi, C. Marchi: Multiscale problems and homogenization for second-order Hamilton-Jacobi equations, J. Differential Equations 243 (2007), 349--387.
\bibitem{A1} 
M. Arisawa,  {\em Homogenization of a Class of Integro-Differential Equations with L\'evy Operators.} Comm. Partial Differential Equations, 34: 617-624 (2009).

\bibitem{A2} 
M. Arisawa,  {\em Homogenization of Integro-Differential Equations with L\'evy Operators with Asymmetric and Degenerate Kernels.} Proc. Royal Soc. of Edinburgh 142A, 917-943 (2012).

\bibitem{A3}
M. Arisawa, {\em Quasi periodic and almost periodic homogenizations of integro-differential equations with L\'evy operators.} Preprint 2011, https://arxiv.org/abs/1111.1042.
 
\bibitem{B-CD}
M. Bardi, I. Capuzzo-Dolcetta,  {\em Optimal Control and Viscosity Solutions of Hamilton-Jacobi-Bellman Equations.}  
Birkh\"auser , Boston, 1997.

\bibitem{Barlesbook}
G. Barles,  {\em Solutions de Viscosite des Equations de Hamilton-Jacobi} Collection ``Mathematiques et Applications'' de la SMAI, n. 17, 
Springer-Verlag (1994).
 
 
\bibitem{B-Ch-C-I-Lip}
G. Barles, E. Chasseigne, A. Ciomaga, C. Imbert, {\em Lipschitz Regularity of Solutions for Mixed Integro-Differential Equations.}
J. Differential Equations 252 (2012), 6012-6060.
%


\bibitem{BI}
G. Barles, C. Imbert, {\em Second-order Elliptic Integro-Differential Equations: Viscosity Solutions' Theory Revisited.}  Ann. Inst. H. Poincar\'e Anal. Non Lin\'eare, 25 (2008), 
 567-585.
%
\bibitem{BKLT}
G. Barles,  S. Koike, O. Ley, E. Topp,  
{\em Regularity Results and Large Time Behavior for Integro-Differential Equations with Coercive Hamiltonians.} 
Calc. Var. Partial Differ. Eq., 54, 
 539-572 (2015).

\bibitem{BLT}
G. Barles, O. Ley, E. Topp,  {\em Lipschitz Regularity for Integro-Differential Equations with Coercive Hamiltonians and Application to Large Time Behavior.}
 Nonlinearity 30 (2017), no. 2, 703--734.
 
\bibitem{BP1}
G. Barles, B. Perthame, {\em Exit time problems in optimal
control and vanishing viscosity method.} SIAM J. Control 
Optim. 26 (1988), 1133--1148.




\bibitem{FBRS}
J. Fern\'andez Bonder, A. Ritorto, A. Salort,  {\em $H$-convergence results for nonlocal elliptic-type problems via Tartar's method.} SIAM J. Math. Anal. 49 (2017), 
 2387--2408. 

\bibitem{CS1}
L. Caffarelli, L. Silvestre, {\em Regularity theory for {   fully nonlinear} 
integro-differential equations.} Comm. Pure Appl. Math, 62 (2009), 
 597-638.

%
%
%
%
\bibitem{ChD}
Chang-Lara, H., D\'avila, G. {\em Regularity for solutions of nonlocal, nonsymmetric equations.} Ann. Inst. H. Poincar\'e Anal. Non Lin\'eaire 29 (2012), 
833-859. 

 \bibitem{C}
 A. Ciomaga, {\em On the Strong Maximum Principle for Second Order Nonlinear Parabolic Integro-Differential Equations} Advances in Diff. Equations.
 17 (2012), 635-671.

\bibitem{DQT3}
G. D\'avila, A. Quaas, E. Topp, {\em Existence, nonexistence and multiplicity results for fully nonlinear nonlocal Dirichlet problems.} J. Differential Equations (2018) https://doi.org/10.1016/j.jde.2018.10.046
 
 \bibitem{Hitch}
E. Di Nezza, G. Palatucci, E. Valdinoci,  {\em Hitchhiker's Guide to the Fractional Sobolev Spaces.}
 Bull. Sci. Math., 136, (2012), 
  521--573.
%
\bibitem{Evans1}
L.C. Evans, \emph{The perturbed test function method for viscosity solutions of nonlinear PDEs} Proc. R. Soc. Edinb. A 111 (1989), 359--375.

\bibitem{Evans2}
L.C. Evans, \emph{Periodic Homogeneization of certain fully nonlinear partial differential equations} Proc. R. Soc. Edinb. A 120 (1992), 245--265.

\bibitem{LPV}
P.-L. Lions, G. Papanicolaou, S.R.S. Varadhan, {\em Homogeneization of Hamilton-Jacobi Equations.} Manuscript, 1986. 

\bibitem{PZ} Piatnitski, A.; Zhizhina, E., {\em Periodic homogenization of nonlocal operators with a convolution-type kernel.} SIAM J. Math. Anal. 49 (2017), 
64--81. 

\bibitem{S}
R.W. Schwab,  \emph{Periodic homogenization for nonlinear integro-differential equations.} SIAM J. Math. Anal. 42 (2010), 
 2652--2680.

\bibitem{S2}
R.W. Schwab,  \emph{ Stochastic homogenization for some nonlinear integro-differential equations.} Comm. Partial Differential Equations 38 (2013), 
 171--198.
\bibitem{Si1}
L. Silvestre, {\em Regularity of the obstacle problem for a fractional power of the Laplace operator.} Comm. Pure Appl. Math. 60 (2007), 67--112.

\bibitem{Silvestre}
L. Silvestre, {\em On the Differentiability of the Solution to the Hamilton-Jacobi Equation with Critical Fractional Diffusion.} Adv. Math. 226 (2011), no. 2, 2020--2039. 



%
%
%
%
\end{thebibliography}
\end{document}